\documentclass[a4paper,11pt,reqno]{amsart}
\usepackage{amsthm}
\usepackage{amsmath}
\usepackage{enumerate}
\usepackage{amsfonts}
\usepackage{amssymb}
\usepackage{fullpage}
\usepackage{amsmath,amscd}
\usepackage{stmaryrd}
\usepackage{graphicx}
\usepackage{array}
\usepackage{amsmath,amssymb,amsfonts,dsfont}
\usepackage[utf8]{inputenc} 
\usepackage[english]{babel}
\usepackage[T1]{fontenc} 
\usepackage{graphicx}
\usepackage{float}
\usepackage{pdfpages}
\usepackage[a4paper, margin = 3cm, bottom = 3cm]{geometry}
\usepackage{ifpdf}
\usepackage{marginnote}
\usepackage{hyperref}
\usepackage{tikz}
\hypersetup{pdfborder=0 0 0, 
	    colorlinks=true,
	    citecolor=black,
	    linkcolor=blue,
	    urlcolor=red,
	    pdfauthor={Guillaume Tahar}
	   }
\newtheorem{thm}{Theorem}[section]
\newtheorem{cor}[thm]{Corollary}
\newtheorem{prop}[thm]{Proposition}
\newtheorem{lem}[thm]{Lemma}
\theoremstyle{definition}
\newtheorem{defn}[thm]{Definition}
\theoremstyle{remark}
\newtheorem{rem}[thm]{Remark}
\theoremstyle{definition}

\theoremstyle{definition}
\newtheorem{ex}[thm]{Example}
\theoremstyle{definition}

\numberwithin{equation}{section} 
\title{Counting saddle connections in flat surfaces with poles of higher order}
\author{Guillaume Tahar} 
\address[Guillaume Tahar]{Institut de Math{\'e}matiques de Jussieu - UMR CNRS 7586}
\email{guillaume.tahar@imj-prg.fr}

\date{December 6, 2017}
\begin{document}
\begin{abstract}
Flat surfaces that correspond to $k$-differentials on compact Riemann surfaces are of finite area provided there is no pole of order $k$ or higher. We denote by \textit{flat surfaces with poles of higher order} those surfaces with flat structures defined by a $k$-differential with at least one pole of order at least $k$. Flat surfaces with poles of higher order have different geometrical and dynamical properties than usual flat surfaces of finite area. In particular, they can have a finite number of saddle connections. We give lower and upper bounds for the number of saddle connections and related quantities. In the case $k=1\ or\ 2$, we provide a combinatorial characterization of the strata for which there can be an infinite number of saddle connections.\newline
\end{abstract}
\maketitle
\setcounter{tocdepth}{1}
\tableofcontents

\section{Introduction}

Translation surfaces correspond to nonzero holomorphic 1-forms (Abelian differentials) on a compact Riemann surface. In \cite{Bo}, a \textit{translation surface with poles} is defined by a meromorphic $1$-form on a compact Riemann Surface, where the poles correspond to the punctured points. In this paper, we also consider $k$-differentials, that are differentials of the form $f(z)dz^{k}$. Nonzero $k$-differentials define flat metrics as well and their holonomy can be nontrivial (rotations of angle $\dfrac{2m\pi}{k}$).\newline
Zeroes of arbitrary order and poles of order at most $k-1$ of a meromorphic $k$-differential are the conical singularities of the flat metric induced. Poles of order at least $k$ are referred to as \textit{poles of higher order}. Every neighborhood of a pole of higher order is of infinite area for the induced flat structure. In this paper, we mainly focus on saddle connections of flat surfaces with poles of higher order. A saddle connection is a geodesic segment joining two conical singularities and having no singularity inside it.\newline
Although we are mainly interested in $1$-forms ($k=1$) and quadratic differentials ($k=2$) many of our constructions holds for arbitrary $k$-differentials.\newline
Translation surfaces (corresponding to holomorphic $1$-forms) and half-translation surfaces (corresponding to meromorphic quadratic differentials with at most simple poles) have been deeply studied, both from a geometrical and a dynamical point of view, see \cite{Zo}. They are flat surfaces of finite area.\newline
Flat surfaces that correspond to meromorphic $k$-differentials have finite genus and their interest is much more geometric than dynamic. They have been connected with at least three different fields of study:\newline
- wall-crossing structures in spaces of stability conditions on triangulated categories, see \cite{BS,GMN,KS}.\newline
- spectral families of linear differential operators, see \cite{Sh}.\newline
- compactifications of strata of surfaces of finite area, see \cite{Au, BCGGM, EKZ, FP, Ge, MZ}, and invariant submanifolds of these strata, see \cite{EMM,F}.\newline

In this paper, we initiate the study of the counting problem in the case of strata of $k$-differentials with poles of higher order. Our objective is twofold. First, we try to solve the problem as completely as possible for meromorphic $1$-forms and meromorphic quadratic differentials. As it turns out, many of our intermediate results equally hold for differentials of higher order and may be used as foundational results for a further study.\newline
\newline

\section{Statement of main results}

Let $X$ be a compact Riemann Surface. $K$ is the canonical line bundle on $X$. A meromorphic $k$-differential $\phi$ is a meromorphic section of $K^{\otimes k}$. When $\phi$ has at least one pole of order at least $k$, then $(X,\phi)$ is called a \textit{flat surface with poles of higher order}. When there is no ambiguity, it will be referred to as $X$. Flat surfaces with poles of higher order are infinite area surfaces with "finite complexity" \cite{Bo1, HKK}.\newline
In the following, without further specification, we study differentials with at least one pole of order at least $k$. If $k=1$, we get translation surfaces with poles studied in \cite{Bo}.\newline

Zeroes and poles of order at most $k-1$ of the $k$-differential are conical singularities for the induced flat metric \cite{BCGGM1, St}. We denote by $\mathcal{H}^{k}(a_1,\dots,a_n,-b_1,\dots,-b_p)$ the stratum in the moduli space of meromorphic $k$-differentials that correspond to differentials with conical singularites of degrees $a_1,\dots,a_n \in \lbrace-k+1,\dots,-1\rbrace \cup \mathbb{N}^{\ast}$ and poles of degrees $b_1,\dots,b_p \geq k$. Family $a_1,\dots,a_n,-b_1,\dots,-b_p$ is the \textit{singularity pattern} of the stratum. The integer $g$ such that $\sum \limits_{i=1}^n a_i - \sum \limits_{j=1}^p b_j = k(2g-2)$ is the \textit{genus} of the flat surface with poles of higher order.\newline

Unless otherwise indicated, we will assume that $n>0$ and $p>0$. Besides, if $g=0$, we will assume that $n+p \geq 3$ because there are no saddle connections in surfaces with too few singularities. In order to simplify our study, there are no marked points in the surfaces we study.\newline

\subsection{Reducibility index and graph representations} 

A first investigation about dynamics of flat surfaces of infinite area shows that an infinite number of saddle connections is equivalent to existence of a cylinder of finite area in the flat surface. These cylinders are spanned by closed geodesics (we always assume cylinders are nondegenerate and have positive area). In the following, we provide a combinatorial characterization of strata whose flat surfaces may contain such cylinders. Since closed geodesics may intersect themselves when $k\geq3$, we focus on the case $k=1\ or\ 2$.\newline

On that assumption, cutting along a family of closed geodesics decomposes a surface into especially nice components connected by cylinders according to a graph. Such a decomposition defines what we denote by a \textit{graph representation}, that is a graph where vertices correspond to connected components and edges correspond to cylinders. Every singularity of the singularity pattern $a_1,\dots,a_n,-b_1,\dots,-b_p$ of a stratum $\mathcal{H}^{k}(a_1,\dots,a_n,-b_1,\dots,-b_p)$ is assigned to a vertex, see Figure 1. In the following, we define abstractly what a graph representation is. Lemma 6.1 ensures that what we get by cutting along a family of closed geodesics is indeed a graph representation.\newline

\begin{figure}
\includegraphics[scale=0.3]{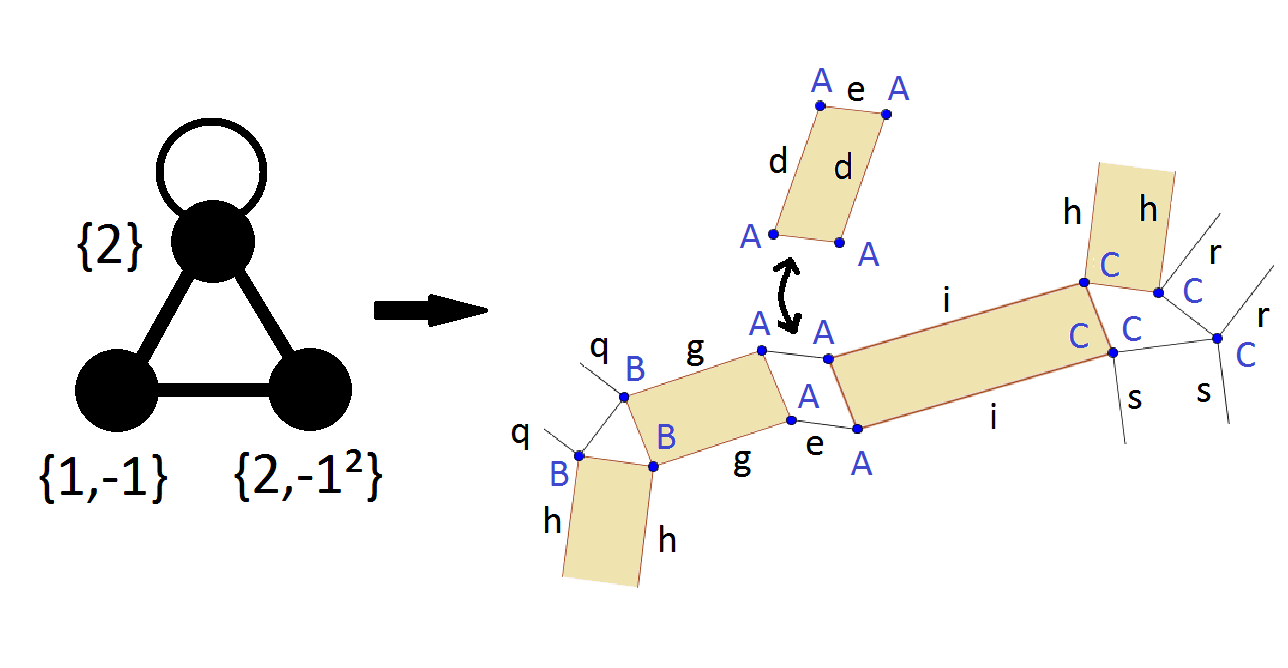}
\caption{A graph representation of $\mathcal{H}^{1}(2^{2},1,-1^{3})$ and a flat realization. The two zeroes of order two are $A$ and $C$ whereas the zero of order one is $B$.}
\end{figure}

\begin{defn}
Let $\mathcal{H}^{k}(a_1,\dots,a_n,-b_1,\dots,-b_p)$ be a stratum of $k$-differentials on a surface of genus $g$ such that $k=1\ or\ 2$ and $s$ is a nonnegative integer. A \textit{graph representation of level $s$} of the stratum is defined by the following data $(G,f_0,\dots,f_s,w_0,\dots,w_s)$. $G$ is a connected multigraph (there can be more than one edge between two vertices) with $s+1$ vertices of valencies $v_0,\dots,v_s$ and with $t$ edges. Nonnegative integers $w_0,\dots,w_s$ are weights associated to each vertex (they correspond to the genus of each connected component in the decomposition). Numbers $a_1,\dots,a_n,-b_1,\dots,-b_p$ are split into a partition $f_{0},\dots,f_{s}$ of families that correspond to the subsets of zeroes and poles assigned to each connected component in the decomposition. They have to satisfy the following conditions:\newline
(i) there is at least one number among $(a_1,\dots,a_n)$ in each family $f_i$ (there must be a conical singularity in each component),\newline
(ii) the sums $\sigma(i)$ of the numbers in each family $f_i$ satisfy $\sigma(i) =k.v_{i}+2k.w_{i}-2k$,\newline
(iii) only in the case $k=1$, if cutting an edge splits the graph into two connected components, then there is a number among $(-b_1,\dots,-b_p)$ in at least one family of each connected component.\newline

As a consequence of formula expressed in condition (ii), we have $t=s+g-\sum \limits_{i=0}^s w_i$.\newline
Besides, we have $0 \leq \sum \limits_{i=0}^s w_i \leq g$.\newline
\newline
A \textit{graph representation} is \textit{pure} when weights $w_0,\dots,w_s$ are identically zero. In this case, $\dfrac{\sigma}{k}$ is the canonical divisor of graph $G$, see \cite{Ba}.\newline
One can always transform a graph representation $(G,f_0,\dots,f_s,w_0,\dots,w_s)$ into a pure graph representation by attaching to every vertex a number of loops equal to its weight and reducing every weight to zero.\newline
\end{defn}

In the following, we distinguish strata for which there is a combinatorial obstruction to existence of cylinders.

\begin{defn}
Strata of genus zero whose graph representations are all of level zero are called {irreducible strata}. Any other stratum is a \textit{reducible stratum}.\newline
In particular, for strata of genus zero, a necessary and sufficient condition for being an irreducible stratum is that we cannot split the singularity pattern into two families such that the sum of the numbers in each family is $-k$.\newline
\end{defn}

In Theorem 2.3, we prove that the only obstruction is combinatorial. Every combinatorially possible graph representation is actually realized by some flat surface in the stratum.\newline

\begin{thm}
For every graph representation $(G,f_0,\dots,f_s,w_0,\dots,w_s)$ of a stratum $\mathcal{H}$, there is a flat surface $X$ and a family $F$ of cylinders such that $(X,F)$ realizes the graph representation $(G,f_0,\dots,f_s,w_0,\dots,w_s)$ of the stratum $\mathcal{H}$.\newline
Besides, we can choose $(X,F)$ such that the closed geodesics of all cylinders of the family $F$ share the same direction.\newline
\end{thm}

Theorem 2.3 implies a complete characterization of strata where every surface has finitely many saddle connections.

\begin{cor} Every flat surface in a stratum has finitely many saddle connections if and only if the stratum is irreducible (in particular, $g=0$).
\end{cor}

It must be noted that in every stratum of meromorphic $1$-forms or quadratic differential with poles of higher order, we can find flat surfaces which have finitely many saddle connections (Proposition 7.1).\newline

Theorem 2.3 is proved in section 6.\newline

\subsection{Core and chambers}
 
Although flat surfaces with poles of higher order have infinite area, most of their geometry is encompassed in a finitary structure. The \textit{core} of a flat surface with poles of higher order $(X,\phi)$ is the convex hull $core(X)$ of the set of conical singularities of the $k$-differential for the locally Euclidian metric, see Definition 4.1 for details. The concept of core has been introduce in \cite{HKK} in the case of meromorphic quadratic differentials.\newline

Structural changes of the core define a walls-and-chambers structure on the strata, see Definition 4.12. Every topological property of the core is invariant in chambers, see Proposition 4.15.\newline

In Theorem 2.5, we prove a dichotomy between chambers where some surfaces have an infinite number of saddle connections and chambers where the number of saddle connections is globally bounded. In particular, the following theorem forbids existence of a chamber where every surface would have a finite number of saddle connections but where this number could be arbitrarily large. Theorem 2.5 is only about strata of $1$-forms. Since closed geodesics are very different if $k\geq3$, a generalization of this theorem would require deeper insights. Such a theorem would be easier in the case $k=2$ but we would need better combinatorial information about the core in order to get a sharp bound.\newline

We crucially use flat triangulations of the core (triangulations whose edges are saddle connections). For a given flat surface $(X,\phi)$, the number $t$ of ideal triangles in a flat triangulation of $core(X)$ is constant in the chamber $(X,\phi)$ belongs to, see Lemma 4.10. Besides, the numbers $t_1,\dots,t_s$ of ideal triangles in every connected component of the interior of $core(X)$ are also constant in the chamber. We have $t=\sum \limits_{{i=1}}^s t_i$.

\begin{thm}
For every stratum $\mathcal{H}^{1}(a_1,\dots,a_n,-b_1,\dots,-b_p)$ of flat surfaces of genus $g$ and for every chamber $\mathcal{C}$, exactly one of these two statements is true:\newline
- $\mathcal{C}$ is a \textit{chamber of bounded type} : for every surface $(X,\phi)$ in chamber $\mathcal{C}$, we have $|SC| \leq 2g-2+n+p+t+\dfrac{1}{2}\sum \limits_{{i=1}}^s t_{i}(t_{i}-1)$.\newline
- $\mathcal{C}$ is a \textit{chamber of unbounded type} : there is a surface $(X,\phi)$ in chamber $\mathcal{C}$ such that $|SC|=+\infty$.\newline
\end{thm}

Since there are $\dfrac{d(d-1)}{2}$ geodesic segments in a convex flat $d$-gon, we can construct translation surfaces with poles (in the case $k=1$) whose core is a convex flat $d$-gon and that realizes the case of equality of the bound, see Figure 8 for an example. Therefore, our bound is sharp in a certain way.\newline

Theorem 2.5 provides a bound on the number of saddle connections for translation surfaces with poles (when $k=1$) that belong to chambers where no surface has an infinite number of saddle connections. In particular, for irreducible strata, it implies a bound that holds in the whole stratum.

\begin{cor}
Every chamber in a stratum $\mathcal{H}^{1}$ is of bounded type if and only if $\mathcal{H}$ is an irreducible stratum. In this case, for every flat surface in $\mathcal{H}$, we have:
$$|SC| \leq 2g-2+n+p+\dfrac{(4g-4+2n+p)(4g+2n+p-3)}{2}$$.
\end{cor}

Theorem 2.5 and Corollary 2.6 are proved in section 8.\newline

The structure of the paper is the following: \newline
- In Section 3, we recall the background about flat surfaces with poles of higher order: flat metric, saddle connections, moduli space and their components. \newline
- In Section 4, we introduce one of our main tools: the core of a flat surface with poles and its associated wall-and-chambers structure. We also introduce the notion of maximal geodesic arc system.\newline
- In Section 5, we prove some preliminary dynamical results. Then, we describe the leaves of the vertical foliation, obtaining a decomposition theorem for directions of saddle connections. We characterize the accumulation points of directions of saddle connections as directions of finite volume cylinders or minimal components. We end with a systematic construction of surfaces with a degenerate core.\newline
- In Section 6, we introduce the notion of reducibility index for strata of $1$-forms and quadratic differentials. We give bounds on the number of finite volume invariant components. Then, we characterize the strata and chambers where some surfaces have an infinity of saddle connections. \newline
- In Section 7, we get lower and upper bounds on the maximal number of noncrossing saddle connections.\newline
- In Section 8, we prove an upper bound on the number of saddle connections in chambers and strata where there is no finite volume cylinder.\newline

\textit{Acknowledgements.} I thank my doctoral advisor Anton Zorich for motivating my work on this paper and many interesting discussions. I am grateful to Corentin Boissy, Quentin Gendron, Elise Goujard and Ben-Michael Kohli for their valuable remarks. This work is supported by the ERC Project "Quasiperiodic" of Artur Avila.\newline

\section{Flat structures defined by meromorphic differentials}

\subsection{Flat structures}

Let $X$ be a compact Riemann surface and let $\phi$ be a meromorphic $k$-differential, that is a differential of the kind $f(z)dz^{k}$. $\Lambda$ is the set of conical singularities of $\phi$. They are the poles of order strictly smaller than $k$ and the zeroes of arbitrary order. $\Delta$ is the set of poles of order at least $k$ of $\phi$. They will be referred to as \textit{poles of higher order}.\newline

Outside $\Lambda$ and $\Delta$, $\phi$ is locally the $k^{th}$ power of a holomorphic differential $\omega$. Integration of $\omega$ in a neighborhood of $z_{0}$ gives local coordinates whose transition maps are of the type $z \mapsto e^{i\theta}z+c$ where $\theta\in\lbrace 0,\dfrac{2\pi}{k}\dots,(k-1)\dfrac{2\pi}{k} \rbrace$.\newline

The pair $(X,\phi)$ seen as a smooth surface with such an atlas is called a \textit{flat surface with poles of higher order}.\newline

In a neighborhood of a zero of order $a>0$, the metric induced by $\phi$ admits a conical singularity of angle $(1+\dfrac{a}{k})2\pi > 2\pi$. In a neighborhood of a pole of order $1\leq a \leq k-1$, the metric induced by $\phi$ admits a conical singularity of angle $(1-\dfrac{a}{k})2\pi < 2\pi$, see Strebel in \cite{St} for quadratic differentials and \cite{BCGGM1} for general $k$-differentials.\newline

\subsection{Local model for poles of higher order} 

In a neighborhood of a pole of order $k$, the $k$-differential is of the form $\dfrac{r}{z^{k}}(dz)^{k}$ where $r$ is called the residue at the pole. The neighborhood of a pole of order $k$ is a semi-infinite cylinder with one end. The residue of the differential is the $k^{th}$ power of the holonomy vector of the closed geodesics defined by this cylinder, considered up to a rotation of angle $\dfrac{2\pi}{k}$. This shows that the residue at a pole of order $k$ is always nonzero, see section 3 in \cite{BCGGM1}. Besides, it shows that the residue is the only local invariant of the pole.\newline

In a neighborhood of a pole of order $b>k$, the differential is of the form $
(z^{\frac{b}{k}}+\dfrac{s}{z})^{k}(dz)^{k}$ where $s^{k}$ is the residue at the pole. Once more the residue is the only local invariant of the pole, see section 3 in \cite{BCGGM1} for details.\newline

A flat cone is the flat surface of genus zero with one conical singularity of order $a$ and one pole of order $b$ such that $a-b=-2k$. In the case $a=0$, there is no conical singularity but there is a unique pole of order $2k$. This surface is the flat plane (the pole of order $2k$ is the point at infinity).\newline
\newline
The neighborhood of a pole of order $b>k$ with trivial residue is obtained by taking an infinite cone of angle $(\dfrac{b}{k}-1)2\pi$ and removing a compact neighborhood of the conical singularity.\newline
\newline
In order to get neighborhoods of poles of order $b>k$ with non trivial residue, we follow Boissy in \cite{Bo}. We take a flat cone of angle $(\dfrac{b}{k}-1)2\pi$ and remove an $\epsilon$-neighborhood of a semi-infinite line starting from the conical singularity and a neighborhood of the conical singularity. Then, we identify the resulting boundaries of the neighborhood by an isometry. Rotating and rescaling gives a pole of order $k$ with the residue we want.\newline

\begin{lem} Every pole of higher order has a basis of neighborhoods whose complement is convex.
\end{lem} 

\begin{proof} Neighborhoods of poles of order $k$ are semi-infinite cylinders and thus are spanned by a family of parallel closed geodesics. We can pick a converging family of these geodesics as boundaries of neighborhoods of a basis.\newline
We follow the classical construction of neighborhoods of poles of order $b > k$. We consider a conical singularity of order $2k-b$, draw a small disk $V$ of this zero. After a change of coordinates $w=1/z$, the complement to $V$ is a convex neighborhood of the pole as required for the flat metric. The neighborhoods of poles having same degrees and same residues are isometric so the lemma is proved in all cases. 
\newline
\end{proof}

\subsection{Moduli space}

If $(X,\phi)$ and $(X',\phi')$ are flat surfaces such that there is a biholomorphism $f$ from $X$ to $X'$ such that $\phi$ is the pullback of $\phi'$, then $f$ is an isometry for the flat metrics defined by $\phi$ and $\phi'$.\newline

As in the case of Abelian differentials, we define the moduli space of meromorphic $k$-differentials as the space of equivalence classes of flat surfaces with poles of higher order $(X,\phi)$ up to biholomorphism preserving the differential.\newline

We denote by $\mathcal{H}^{k}(a_1,\dots,a_n,-b_1,\dots,-b_p)$ the \textit{stratum} that corresponds to meromorphic $k$-differentials with singularities degrees $a_1,\dots,a_n,-b_1,\dots,-b_p$. A necessary condition for a stratum to be nonempty is that $\sum \limits_{i=1}^n a_i - \sum \limits_{j=1}^p b_j = k(2g-2)$ for some integer $g \geq 0$. The integer $g$ is the genus of the underlying Riemann surface. For sake of simplicity we will use the term of strata of genus $g$ to designate strata of differentials that exist only in surfaces of genus $g$.\newline

\subsection{Surgery tools}

There is a surgery called \textit{adding a handle}. We cut a small slit inside a surface and paste the two ends of a cylinder along the two sides of the slit. Such a surgery increases the genus of the surface by one. Besides, it is a local surgery that does not change the geometry of the surface outside a compact subset.\newline

In the case $k=1\ or\ 2$, there is another surgery called \textit{breaking up conical singularities} that starts from a flat surface with a conical singularity of order $a+b$ and provides a flat surface with two conical singularities of orders $a$ and $b$ connected by a short geodesic, see section 2.3 in \cite{La}.\newline
In the case $k=2$, when $a+b$, $a$ or $b$ are odd, the most common construction does not always work. We proceed as follows. We denote by $H$ the conical singularity of order $a+b$. We choose a direction $\theta$ that is the not the direction of a saddle connection. Then, we consider a geodesic trajectory $\alpha$ starting from $H$ and that follows direction $\theta$. Since $\theta$ is the not the direction of a saddle connection, trajectory $\alpha$ converges to a pole of higher order (see Section 5 for details about types of trajectories). Next, we consider a geodesic trajectory $\beta$ starting from $H$ and whose angle at $H$ with $\alpha$ is $(a+1)\pi$. For the same reasons, $\beta$ converges to a pole. Besides, $\alpha$ and $\beta$ follow the same direction $\theta$ so do not cross each other. We cut along trajectories $\alpha$ and $\beta$ and then paste an infinite strip of thickness equal to $\epsilon$ in such a way the angular sector of $H$ is separated into two sectors of angle $(a+1)\pi$ and $(b+1)\pi$ that are respectively above and below the strip. What we get is a surface with two conical singularities of order $a$ and $b$ connected by a short geodesic of length $\epsilon$ that is a transverse segment of the strip. It is crucial to note that this surgery is not local.\newline

\subsection{Properties of strata}

In the case $k=1$, there is a complete characterization of nonempty strata. In addition to the previous condition about genus, a given stratum $\mathcal{H}^{1}(a_1,\dots,a_n,-b_1,\dots,-b_p)$ is nonempty if and only if $\sum \limits_{j=1}^p b_j > 1$. If there are poles, there should not be just one simple pole (because its residue would be zero and the width of the cylinder would be zero either).\newline

In the case $k=2$, when there is no pole of higher order, a stratum is nonempty if and only if $\sum \limits_{i=1}^n a_i = 4g-4$ for some integer $g \geq 0$ and $(a_{1},\dots,a_n)$ is not one of the two following singularity patterns: $(1,-1)$ and $(1,3)$, see \cite{La}.\newline

There is no such empty stratum when there are poles of higher order. This result only use the previously introduced surgery tools.\newline

\begin{prop}
For every singularity pattern $(a_{1},\dots,a_{n},-b_{1},\dots,-b_{p})$ such that we have:\newline
(i) $\sum \limits_{i=1}^n a_i - \sum \limits_{j=1}^p b_j = 4g-4$.\newline
(ii) there is at least one pole of higher order ($p\geq1$).\newline
Then stratum $\mathcal{H}^{2}(a_{1},\dots,a_{n},-b_{1},\dots,-b_{p})$ is nonempty.
\end{prop}

\begin{proof}
Set $\lambda \in \mathbb{C}^{\star}$ and $z_{1},\dots,z_{n},w_{1},\dots,w_{p}$ distinct elements of $\mathbb{C}$. Quadratic differentials of the form $\lambda \prod \limits_{i=1}^n (z-z_{i})^{a_{i}}\prod \limits_{j=1}^p (z-w_{j})^{-b_{j}} dz^{2}$ provide examples in every stratum of genus zero.\newline

Then, we consider the case of strata of genus $g$ with exactly two poles of order two and no poles of higher order. First, we can construct surfaces of strata $\mathcal{H}^{2}(a,-2^{2})$ by adding handles on an infinite cylinder. Then, we break up the unique conical singularity to get surfaces in any stratum $\mathcal{H}^{2}(a_{1},\dots,a_n,-2^{2})$ of genus $g\geq1$.\newline

Next, we consider the case of strata with at least two poles of higher order and such that $\sum \limits_{j=1}^p b_j \geq 5$. As we settled the case $g=0$ when $n+p\geq3$, strata $\mathcal{H}^{2}(-4+\sum \limits_{j=1}^p b_j,-b_{1},\dots,-b_{p})$ is nonempty. Then adding handles and breaking up the unique conical singularity provides surfaces for every stratum with $g\geq1$ and $p\geq2$ (other than those concerned by the previous case).\newline

Finally, we consider the case $p=1$ and $g\geq1$. We first look at strata $\mathcal{H}^{2}(a,-a)$ with $a\geq2$. Figure 2 provides an example of surface in $\mathcal{H}^{2}(2,-2)$. Surfaces in strata $\mathcal{H}^{2}(a,-a)$ with $a\geq3$ are constructed by adding an handle in a flat cone with a pole of order $a$ and a zero of order $a-4$. The case $a=4$ corresponds to the flat plane. Therefore, strata $\mathcal{H}^{2}(a,-a)$ with $a\geq2$ are nonempty. Adding handles and breaking up the unique conical singularity provides surfaces for every stratum with $p=1$ and $g\geq1$. This ends the proof.\newline

\begin{figure}
\includegraphics[scale=0.3]{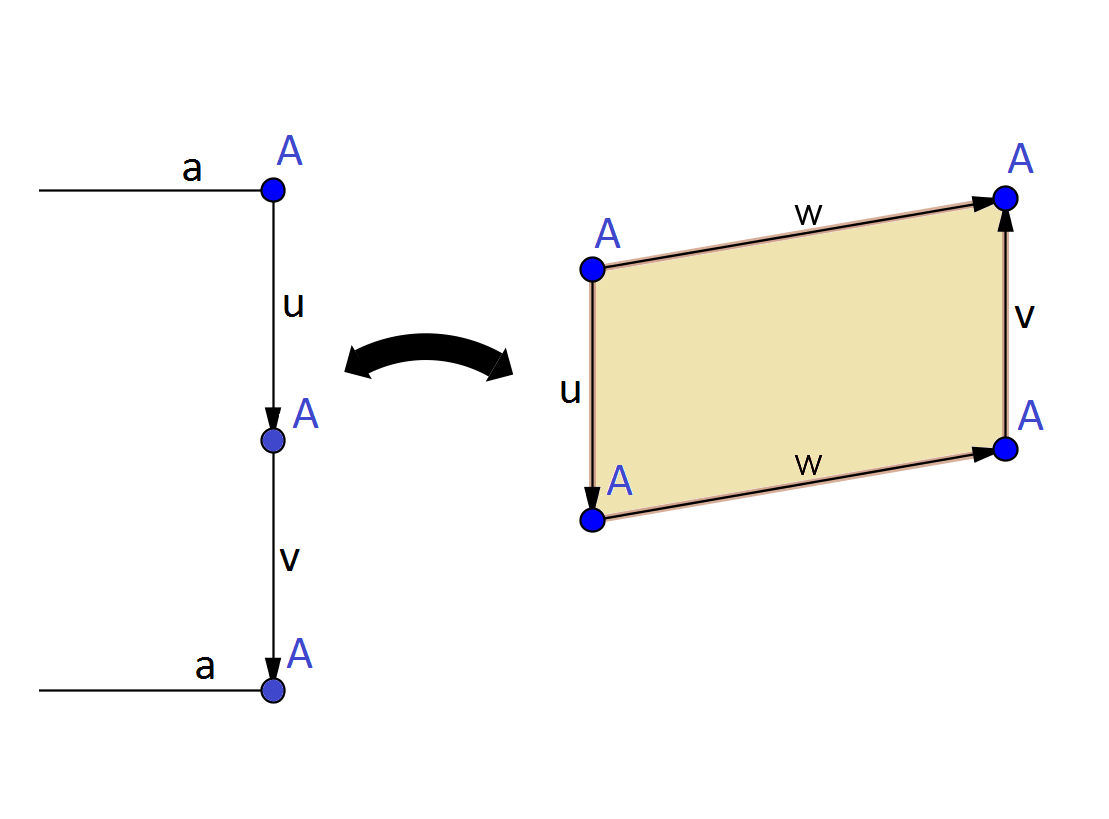}
\caption{A flat surface of $\mathcal{H}^{2}(2,-2)$.}
\end{figure}

\end{proof}

Since the study of $k$-differentials for general $k$ started recently, there is not yet a general characterization of nonempty strata in the case $k\geq3$ and $g\geq1$. 
In the genus zero case, every $k$-differential is of the form $f(z)dz^{k}$ where $f$ is a meromorphic function on the Riemann sphere so any singularity pattern can be represented by a $k$-differential. Therefore, when the genus is zero, there is no empty stratum.\newline

On each stratum of $1$-forms or quadratic differentials, $GL^{+}(2,\mathbb{R})$ acts by composition with coordinate functions \cite{Zo}. As translation surfaces with poles have infinite area, we cannot normalize the area of the surface and thus must consider the full action of $GL^{+}(2,\mathbb{R})$. Since the action of $GL^{+}(2,\mathbb{R})$ does not preserve angles, there is no such action of $GL^{+}(2,\mathbb{R})$ on strata of $k$-differentials when $k\geq3$. Indeed, directions are defined up to a rotation of angle $\dfrac{2\pi}{k}$.\newline

\subsection{Canonical k-cover}

Every $k$-differential (outside the singularities) is locally the $k^{th}$ power of a $1$-form. The canonical $k$-cover assigns to every pair $(X,\phi)$ of a Riemann surface and a $k$-differential a flat surface such that the differential is globally the $k^{th}$ power of a $1$-form, see \cite{BCGGM1}. The cover is branched only over the set of singularities of the differential. The canonical $k$-cover of $(X,\phi)$ is connected if and only if $\phi$ is a \textit{primitive} differential form. A \textit{primitive} differential form is a differential form that is not a power of a differential form of lower order. We denote by $(\widetilde{X},\omega)$ the canonical $k$-cover of $(X,\phi)$.\newline

Differentials that are the global $k^{th}$ power of a $1$-form belong to strata $\mathcal{H}^{k}(kc_{1},\dots,kc_{s})$ where the order of all singularities is divisible by $k$. Their $k^{th}$ root belong to strata $\mathcal{H}^{1}(c_{1},\dots,c_{s})$.\newline

The canonical $k$-cover $(\widetilde{X},\omega)$ of $(X,\phi)$ is cyclic of order $k$ so there is a generator $\tau$ of the cyclic deck group of the cover. Generator $\tau$ induces another periodic automorphism $\tau^{\ast}$ on relative homology group $H_{1}(X\setminus\Delta,\Lambda)$ where $\Delta$ is the set of poles of higher order and $\Lambda$ is the set of conical singularities.\newline
Thus $H_{1}(X\setminus\Delta,\Lambda)$ decomposes into invariant subspaces $H_{\xi_{i}}$ where the $\xi_{i}$ are the $k^{th}$ roots of the unity. $\phi$ belongs to an invariant subspace $H_{\xi_{1}}$ where where $\xi_{1}$ is a primitive root. Strata are complex-analytic orbifolds with local coordinates given by the period map of $H_{\xi_{1}}$ \cite{Bo,BS}.\newline 

\subsection{Saddle connections}

When differentials have poles of higher order, the corresponding surfaces may have a finite number of saddle connections.

\begin{defn} A saddle connection is a geodesic segment joining two conical singularities of the flat surface such that all interior points are not conical singularities.
\end{defn}

It must be noted that when $k\geq3$, saddle connections may be self-intersecting, see Figure 3.\newline

\begin{figure}
\includegraphics[scale=0.3]{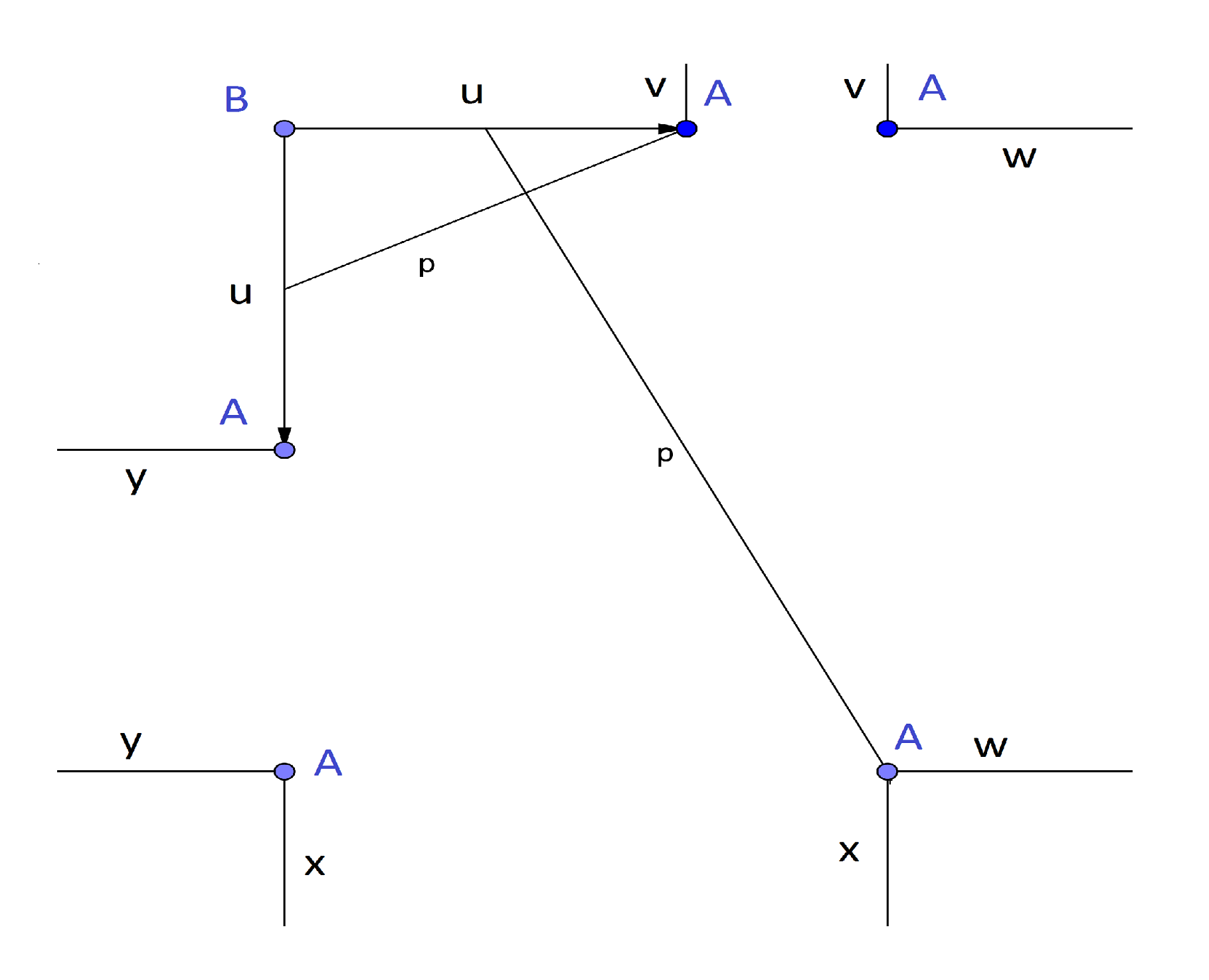}
\caption{A flat surface of $\mathcal{H}^{4}(11,-3,-4^{4})$ with a self-intersecting saddle connection. The zero of order $11$ corresponds to $A$ while the pole of order $3$ (conical singulariy whose total angle is $\dfrac{\pi}{2}$) corresponds to $B$. Saddle connection $p$ connects $A$ with itself.}
\end{figure}

Provided $k=1\ or\ 2$, the number of saddle connections of a flat surface with poles of higher order is $GL^{+}(2,\mathbb{R})-invariant$.

\begin{prop} The number $|SC|$ of saddle connections of a flat surface with poles of higher order $(X,\phi)$ changes lower semicontinuously in the moduli space (the number can only decrease as we pass to the limit).
\end{prop}

\begin{proof} A saddle connection can only be broken if a conical singularity appears in the interior of the geodesic segment over the deformation. So for every saddle connection $\gamma$ there is a neighborhood $U_{\gamma}$ of the surface in the stratum for which the saddle connection $\gamma$ still persists. Let $A$ be a finite set of saddle connections of $X$, then every surface in $\bigcap \limits_{\gamma \in A} U_{\gamma}$ has a number of saddle connections higher than $|A|$.
Therefore one cannot approach a surface by a family of surfaces whose number of saddle connections would not approach it.\newline
\end{proof}

Graphs of noncrossing saddle connections (see maximal geodesic arc systems in Subsection 4.2) provide good combinatorial decompositions of flat surfaces. For this purpose, we need to introduce the concept of \textit{ideal triangle}.

\begin{defn} An \textit{ideal polygon} is a finite area contractible domain of regular points bounded by a finite number of distinct saddle connections whose vertices are not necessarily distinct. An \textit{ideal triangle} is an ideal polygon with three sides.\newline
\end{defn}

We can associate homology classes to saddle connections. If $k\geq2$, the way to do so is quite subtle. We consider a flat surface with poles $(X,\phi)$ and its canonical $k$-cover $(\widetilde{X},\omega)=\pi^{-1}(X,q)$.

$H_{1}(\widetilde{X}\setminus\Delta',\Lambda')$ is the first relative homology group of $(\widetilde{X},\omega)$ where $\Delta'$ and $\Lambda'$ respectively are the preimages by $\pi$ of the conical singularities and poles of higher order associated to $\phi$. A cyclic automorphism $\tau$ associated to the covering acts on $H_{1}(\widetilde{X}\setminus\Delta',\Lambda')$.

We denote by $\gamma_{1},\dots,\gamma_{k}$ the $k$ preimages of an oriented saddle connection $\gamma$. Let $\xi$ be a $k^{th}$ roots of the unity. There is a sequence of integers $l_{1},\dots,l_{k}$ such that $\sum_{i=1}^{k} \xi^{l_{i}}[\gamma_{i}]$ is invariant by the action of $\tau$. We define $[\gamma]=\sum_{i=1}^{k} \xi^{l_{i}}[\gamma_{i}]$.\newline
Two saddle connections are said to be parallel when their relative homology classes are linearly dependant over $\mathbb{R}$.\newline
The holonomy vector of a saddle connection is essentially determined by the period of its relative homology class. Its direction modulo $\dfrac{2\pi}{k}$ is the argument of the period and its length is $\dfrac{1}{k}$ of the modulus of the period.\newline

The following lemma is proved as Theorem 1 in \cite{MZ}. Using the canonical $k$-cover, its proof generalizes without any difficulty in the case of meromorphic $k$-differentials with poles of arbitrary order. Parallel saddle connections occur in the definition of the discriminant of a stratum. Lemma 3.6 is used in particular in Proposition 7.7 for the computation of a lower bound on the number of saddle connections in flat surfaces that lie outside the discriminant.

\begin{lem}
Let $(X,\phi)$ be a flat surface with (or without) poles of higher order. Two saddle connections $\gamma_{1}$ and $\gamma_{2}$ are parallel if and only if they have no interior intersections and one of the bounded connected components of $X \setminus (\gamma_{1}\cup\gamma_{2})$ has trivial linear holonomy.\newline
\end{lem}

\subsection{Index of a loop}

The topological index of a loop is particularly easy to handle in flat geometry.\newline

\begin{defn}
Let $\gamma$ be a simple closed curve in a flat surface with (or without) poles of higher order. $\gamma$ is parametrized by arc-length and passes only through regular points.\newline
We consider the lifting $\eta$ of $\gamma$ by the canonical covering. Then $\eta'(t)=e^{i\theta(t)}$.\newline
We have $\dfrac{1}{2\pi}\int_{0}^{T} \theta'(t)dt \in \dfrac{1}{k}\mathbb{Z}$. This number is the topological index $ind(\gamma)$ of $\gamma$. This number may be fractional because the holonomy of the flat surface may be nontrivial.\newline
\end{defn}

In particular, the topological index of a counterclockwise simple loop around a singularity of order $m$ is $m+1$.\newline

\section{Core and related concepts}

\subsection{Core of a flat surface with poles of higher order}

The core of a flat surface with poles was introduced in \cite{HKK}. It relies on a notion of convexity that is natural in this context but whose consequences may seem counterintuitive.

\begin{defn} A subset $E$ of a flat surface with poles of higher order is \textit{convex} if and only if every element of any geodesic segment between two points of $E$ belongs to $E$.\newline
The convex hull of a subset $F$ of a flat surface with poles of higher order $X$ is the smallest closed convex subset of $X$ containing $F$.\newline
The core of a flat surface with poles of higher order $X$ is the convex hull $core(X)$ of the conical singularities $\Lambda$ of the $k$-differential.\newline
$\mathcal{I}\mathcal{C}(X)$ is the interior of $core(X)$ in $X$ and $\partial\mathcal{C}(X) = core(X)\ \backslash\ \mathcal{I}\mathcal{C}(X)$ is its boundary.\newline
The \textit{core} is said to be degenerate when $\mathcal{I}\mathcal{C}(X)=\emptyset$ that is when $core(X)$ is just graph $\partial\mathcal{C}(X)$.
\end{defn}

We give some examples to illustrate the situations that can occur.

\begin{ex}
For most $X \in \mathcal{H}^{1}(1,-1^{3})$, $\mathcal{I}\mathcal{C}(X) \neq \emptyset$, see Figure 4. We note that according to this definition of convexity, a singleton may be non convex.

\begin{figure}
\includegraphics[scale=0.3]{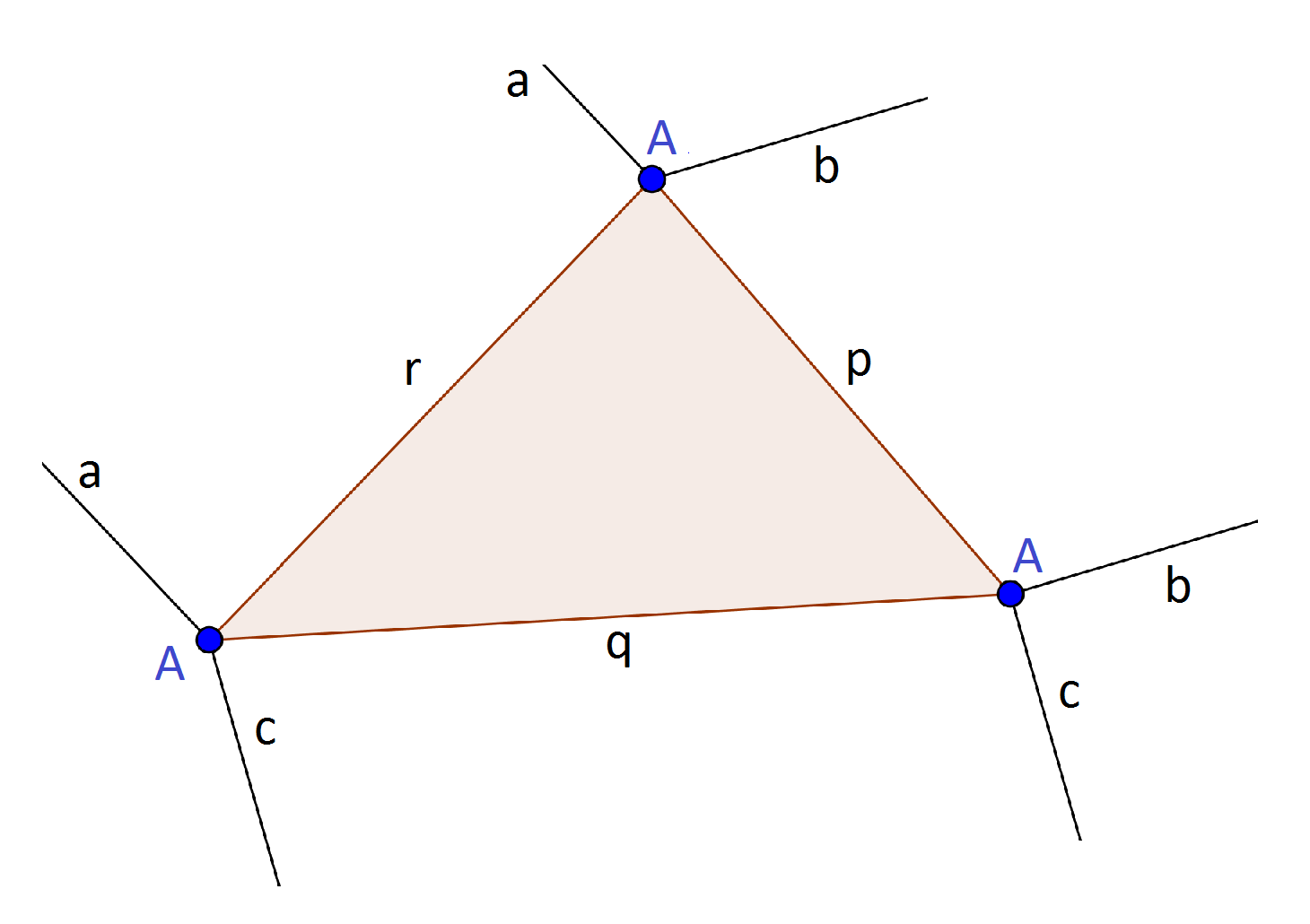}
\caption{The core of the surface $X \in \mathcal{H}^{1}(1,-1^{3})$ as in the picture is an ideal triangle.}
\end{figure}
\end{ex}

\begin{ex}
For some $X \in \mathcal{H}^{1}(2,-2)$, $core(X)$ is degenerate, see Figure 5.

\begin{figure}
\includegraphics[scale=0.3]{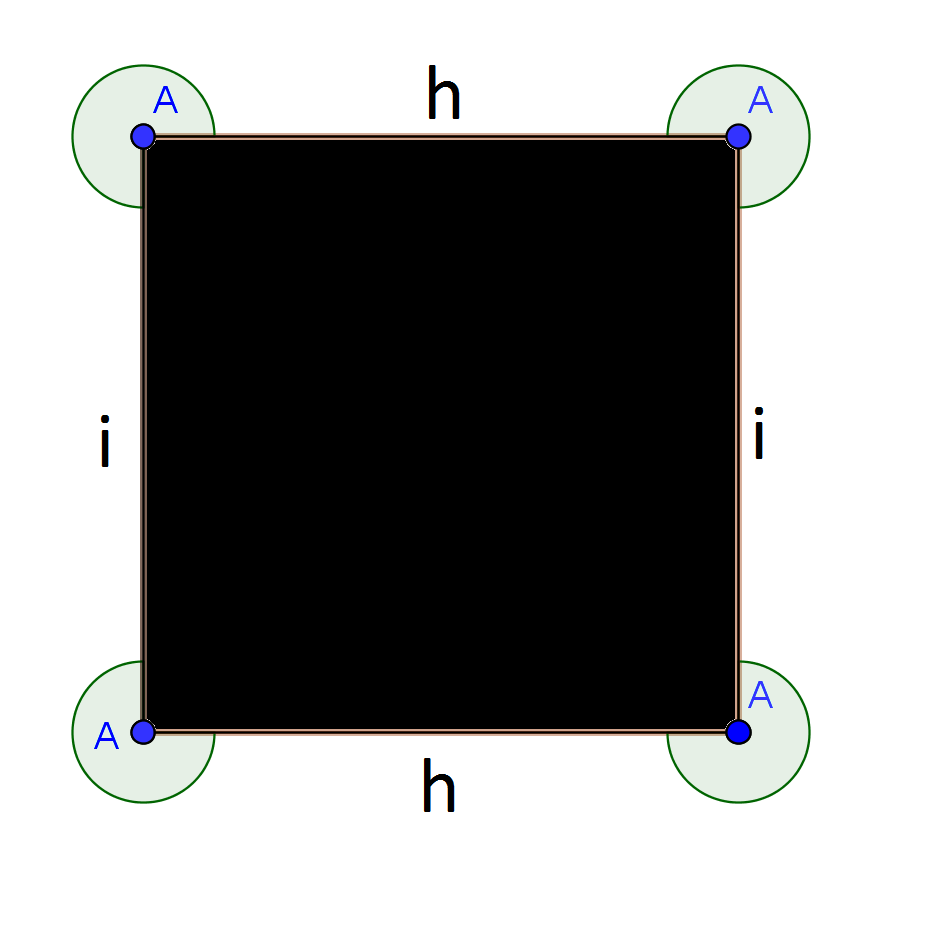}
\caption{The core of this surface from $\mathcal{H}^{1}(2,-2)$ obtained from a plane with a rectangular hole is a bouquet of two circles.}
\end{figure}

\end{ex}

\begin{prop} For any flat surface with poles of higher order $(X,\phi)$, $\partial\mathcal{C}(X)$ is a finite union of saddle connections.
\end{prop}

\begin{proof} Every pole has a neighborhood whose complement is convex (Lemma 3.1). So, $core(X)$ belongs to a compact convex subset of $X$. By hypothesis, a regular point of $\partial\mathcal{C}(X)$ is in the closure of a sequence of geodesic segments that belong to $core(X)$. Therefore, $\partial\mathcal{C}(X)$ is locally isometric to a straight line around regular points. Since the core is compact and the lengths of saddle connections are bounded below by a positive number, $\partial\mathcal{C}(X)$ is a finite union of saddle connections.
\end{proof}

Lemma 4.5 shows that the complement of the core has as many connected components as there are poles of higher order. We refer to these connected components as \textit{domains of poles}.

\begin{lem} Let $X$ be a flat surface with $p$ poles of higher order, then $X \setminus core(X)$ has $p$ connected components. Each of them is a topological disk that contains a unique pole of higher order.
\end{lem}

\begin{proof}
We consider the homotopy class of simple loops around a pole of higher order. A length minimizing representative is the piecewise geodesic boundary of the domain of this pole. Therefore, every pole of higher order belongs to a connected component homeomorphic to a disk.\newline
A connected component without a pole of higher order would be compact so there would be trajectories passing through the core, going into that connected component of the complement of the core and getting back to the core. It is impossible because $core(X)$ is convex. So, there is exactly one pole of higher order in each connected component.\newline
\end{proof}

\subsection{Maximal geodesic arc systems}

In order to get bounds on the number of saddle connections, we examine the maximal number of noncrossing saddle connections that is more tractable. This information is included in the maximal geodesic arc system structure \cite{HKK}.

\begin{defn} A geodesic arc system on a flat surface is a collection of saddle connections that intersect at most at conical singularities. A maximal geodesic arc system (MGAS) is a geodesic arc system that cannot be completed with another saddle connection.
\end{defn}

\begin{rem}
It is possible for $k$-differentials with $k\geq3$ to exhibit saddle connections that intersect themselves, see Figure 3. Therefore, geodesic arc systems include only saddle connections without self-intersection.\newline
\end{rem}

\begin{lem} A maximal geodesic arc system on a flat surface $(X,\phi)$ defines an ideal triangulation of $core(X)$ ($2$-cells of $\mathcal{I}\mathcal{C}(X)$ are ideal triangles). Every conical singularity is a vertex of this triangulation.
\end{lem}

\begin{proof}
Since saddle connections of a geodesic arc system are non-intersecting geodesic segments, they cut out $core(X)$ into several flat surfaces with geodesic boundary. Besides, there are no digons in flat surfaces. Therefore, proving the lemma amount to prove that in any flat surface with geodesic boundary, there is a flat triangulation.\newline
Every relative homology class may be represented by a non-intersecting broken line of saddle connections. We want to cut out every connected component of $\mathcal{I}\mathcal{C}(X)$ into ideal triangles. If the genus of a connected component is nonzero, there is a broken line of saddle connections that represents a nontrivial cycle so we can cut out along saddle connections in order to get components whose genus is zero. Likewise, if there is a connected component whose boundary is not connected, there is a broken line of saddle connections that linking two components of the boundary. In this way, we construct a geodesic arc system such that every $2$-cell is a contractile domain whose boundary is a connected chain of saddle connection. Such cells are plane polygons and it is well known that every plane polygon has a triangulation. It is the same for flat surfaces, see Figure 6.\newline
Every conical singularity is a vertex of the triangulation because every conical singularity can be connected to another zero through a broken line of saddle connections. Each of these connections either belongs to the maximal geodesic arc system or intersects the interior of some saddle connection of the MGAS. Therefore, for every couple of conical singularities, there is a broken line of saddle connections of the MGAS joining them.
\end{proof}

\begin{figure}
\includegraphics[scale=0.3]{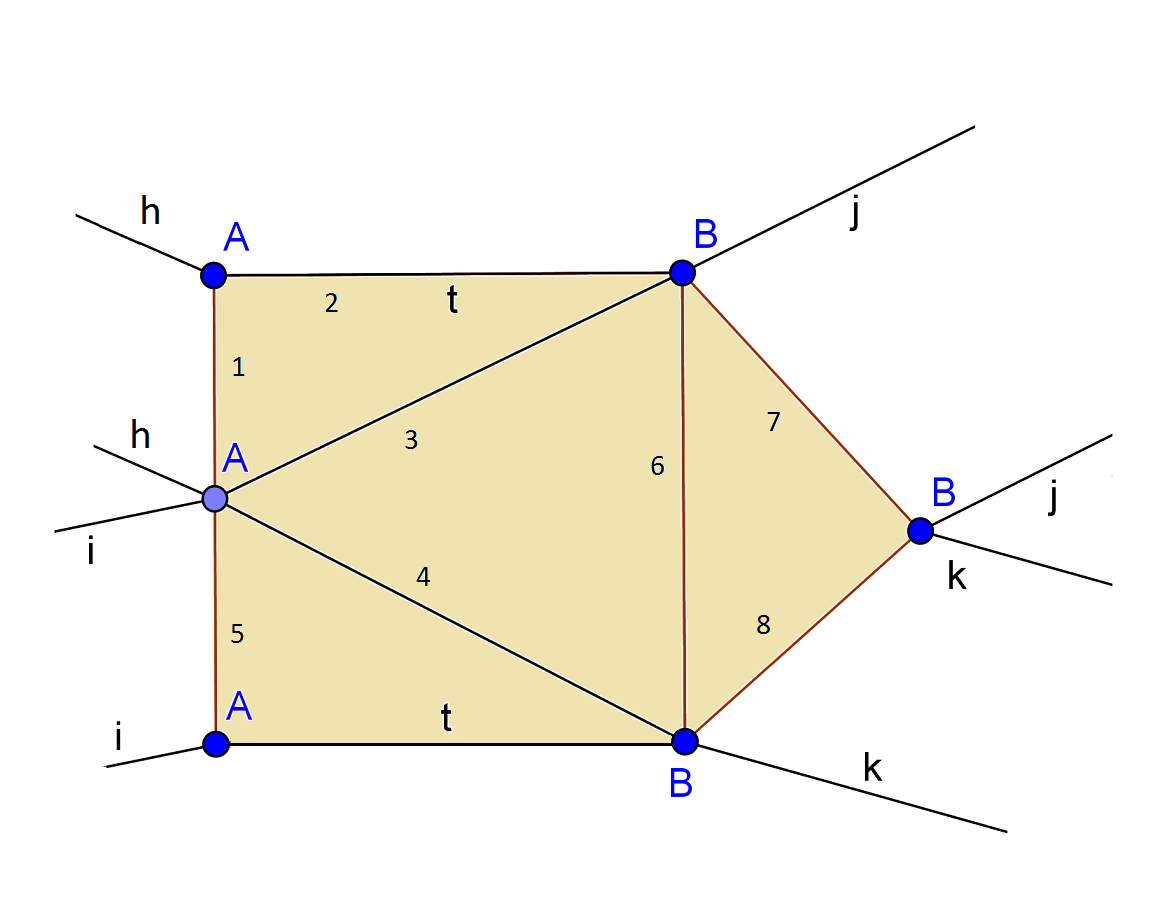}
\caption{A surface in $\mathcal{H}^{1}(1^{2},-1^{4})$ with a maximal geodesic arc system including 8 saddle connections drawn on it}.
\end{figure}

\begin{rem} The saddle connections of $\partial\mathcal{C}(X)$ belong to every MGAS.
\end{rem}

We say that a maximal geodesic arc system defines a partial ideal triangulation of the surface. The 2-dimensional cells are either domains of poles or ideal triangles belonging to the interior of the core. We denote them by \textit{interior faces}. There are $p$ domains of poles and we define $t$ as the number of interior faces.\newline

\begin{lem}
We consider a flat surface with poles of higher order $(X,\phi)$ belonging to $\mathcal{H}^{k}(a_1,\dots,a_n,-b_1,\dots,-b_p)$ and a maximal geodesic arc system with $|A|$ edges and $t$ interior faces of the core. Then we have:
$$|A| = 2g-2+n+p+t$$
\end{lem}

\begin{proof}
Considering the Euler characteristic of the partial ideal triangulation defined by the maximal geodesic arc system, we have:
$$n-|A|+p+t=2-2g$$\newline
that is $$|A|=2g-2+p+n+t$$\newline
\end{proof}

\subsection{Discriminant and walls-and-chambers structure}

The following proposition is crucial to justifiy the definition we give of the discriminant of a stratum, see Figure 5 for an example of a domain of pole where interior angles are strictly greater than $\pi$.

\begin{prop}
Let $\theta_{1},\dots,\theta_{m}$ be the interior angles of the boundary of the domain of a pole of higher order. Then for every $1 \leq i \leq m$, we have $\theta_{i}\geq \pi$.
\end{prop}

\begin{proof}
As $\partial\mathcal{C}(X)$ is a finite union of saddle connections, the boundary of the domain of a pole is a finite union of saddle connections too. We suppose there is an interior angle with a magnitude $0 \leq \theta_{i} < \pi$ between two saddle connections $u$ and $v$ at a conical singularity $A$. If we consider a sufficently small neighborhood of $A$ we can choose two points $B$ and $C$ on $u$ and $v$ such that there is a geodesic segment between $B$ and $C$ belonging to the domain of the pole. This contradicts convexity of $core(X)$.
\end{proof}

Changes of the core define a walls-and-chambers structure on the strata. It should be recalled that two saddle connections are said to be parallel when their relative homology classes are linearly dependent over $\mathbb{R}$, see Lemma 3.6.

\begin{defn} A flat surface with poles of higher order $(X,\phi)$ belongs to the discriminant of the stratum if and only if two nonparallel consecutive saddle connections of the boundary of the core share an angle of $\pi$. Chambers are defined as the connected components of the complement to the discriminant in the strata.\newline
\end{defn}

\begin{ex}
Some flat surfaces $(X,\phi) \in \mathcal{H}^{1}(1,-1^{3})$ belong to the discriminant $\Sigma$ of the stratum, see Figure 7.

\begin{figure}
\includegraphics[scale=0.3]{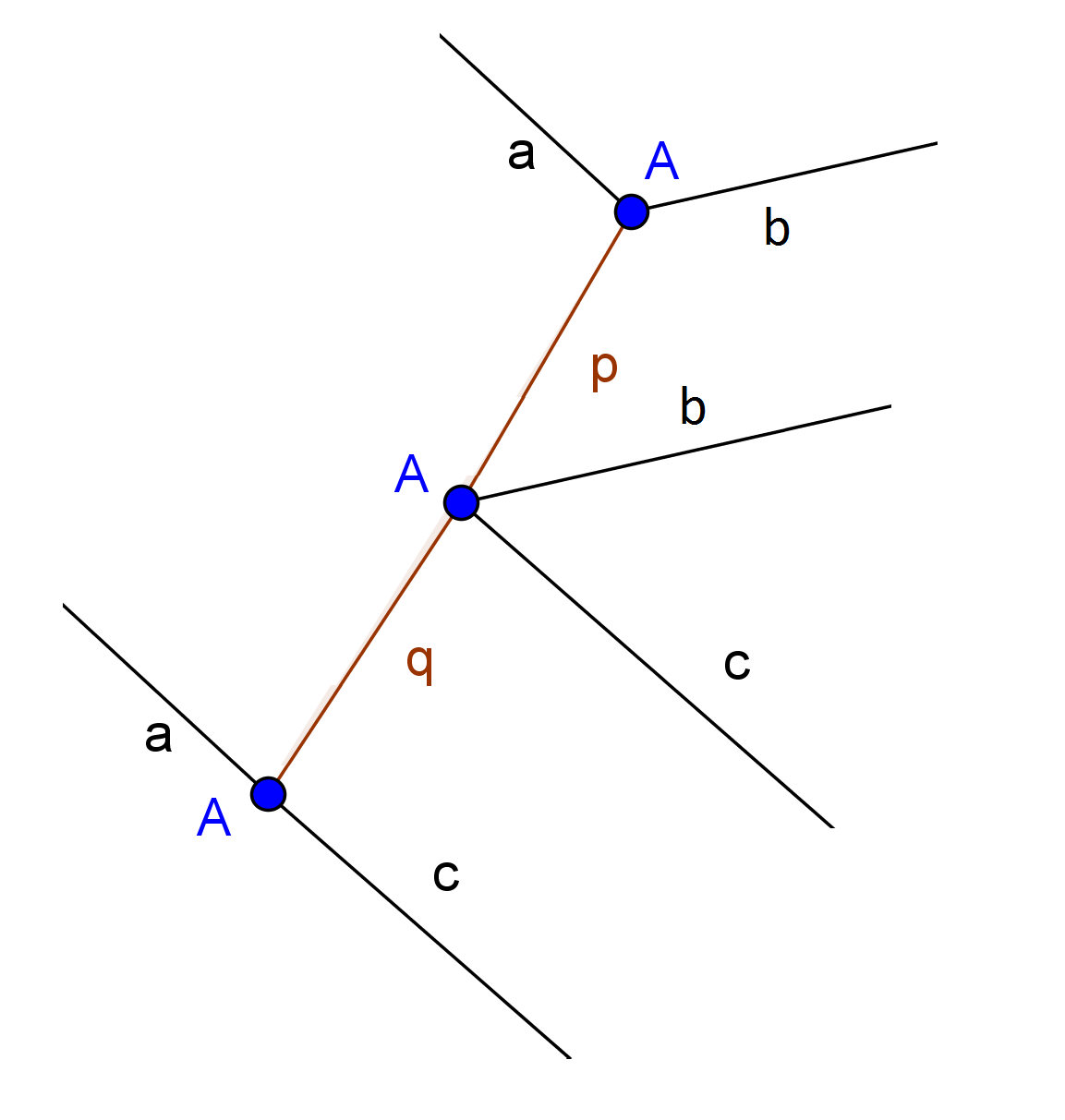}
\caption{Two saddle connections of the boundary of the core sharing an angle of $\pi$.}
\end{figure}

\end{ex}

\begin{prop} The discriminant is an hypersurface of real codimension one in the stratum. If $k=1\ or\ 2$, then it is also invariant under the action of $GL^{+}(2,\mathbb{R})$.
\end{prop}

\begin{proof}
Strata are orbifolds with local coordinates given by the period map. If two homology classes are not parallel, their periods are independent. We interpret periods as vectors of $\mathbb{R}^{2}$. Two vectors are collinear if and only if their determinant is zero. Therefore the locus of real colinearity between the holonomy vectors of two non-parallel saddle connections is of real codimension one. If two non-parallel consecutive saddle connections share an angle of $\pi$, then they are real colinear and the differential belongs to a locus of real codimension one.\newline
The action of $GL^{+}(2,\mathbb{R})$ preserves real colinearity. Therefore, if for a given flat surface there is an angle of $\pi$ between two consecutive saddle connections of the boundary of a domain of pole, then it will also be true for any flat surface in the $GL^{+}(2,\mathbb{R})$-orbit. The action preserves the geometric condition satisfied by the surfaces of the discriminant.\newline
\end{proof}

Many interesting quantities that are constant along the chambers are derived from the topological embedding of the boundary of the core. 

\begin{prop} 
The topological structure on a flat surface with poles of higher order $(X,\phi)$ defined by the embedded graph $\partial\mathcal{C}(X)$ is invariant along the chambers. The number and the degrees of the conical singularities belonging to each connected component of $\mathcal{I}\mathcal{C}(X)$ are invariant too.
\end{prop}

\begin{proof}
Any maximal geodesic arc system on a flat surface defines an ideal triangulation of the core. If a continuous path in the stratum does not lead to degeneration of an ideal triangle, then the topological type of the embedded graph $\partial\mathcal{C}(X)$ in $X$ does not change. Since no saddle connection can shring (otherwise, we would leave the stratum), degeneration of an ideal triangle means an interior angle of $\pi$.\newline
If an ideal triangle whose edges do not belong to $\partial\mathcal{C}(X)$ degenerates, then we can find another maximal geodesic arc system whose triangles are not degenerate. Therefore, the topological type of the embedded graph $\partial\mathcal{C}(X)$ can only change in the case of degeneration of a triangle one of which edge belongs to the boundary of the core. If the interior angle equal to $\pi$ is not opposite to the edge that belongs to the boundary of the core, we can find another maximal geodesic arc system again. Therefore, if a continuous path in the stratum modifies the topological structure, it is because a conical singularity has crossed a saddle connection in the boundary of some domain of pole. Equivalently, the path crosses the discriminant.\newline
\end{proof}

The number of saddle connections is not invariant as we deform a surface inside the chamber, see Figure 8. For these two surfaces, the topological structure of the core is the same. It is an ideal pentagon. The five domains of poles are infinite cylinders glued along the sides of the pentagon. We find that there are less saddle connections in the nonconvex pentagon. It is worth noting that the core is always convex in the sense of the flat metric without necessarily being convex as a plane polygon.\newline

\begin{figure}
\includegraphics[scale=0.3]{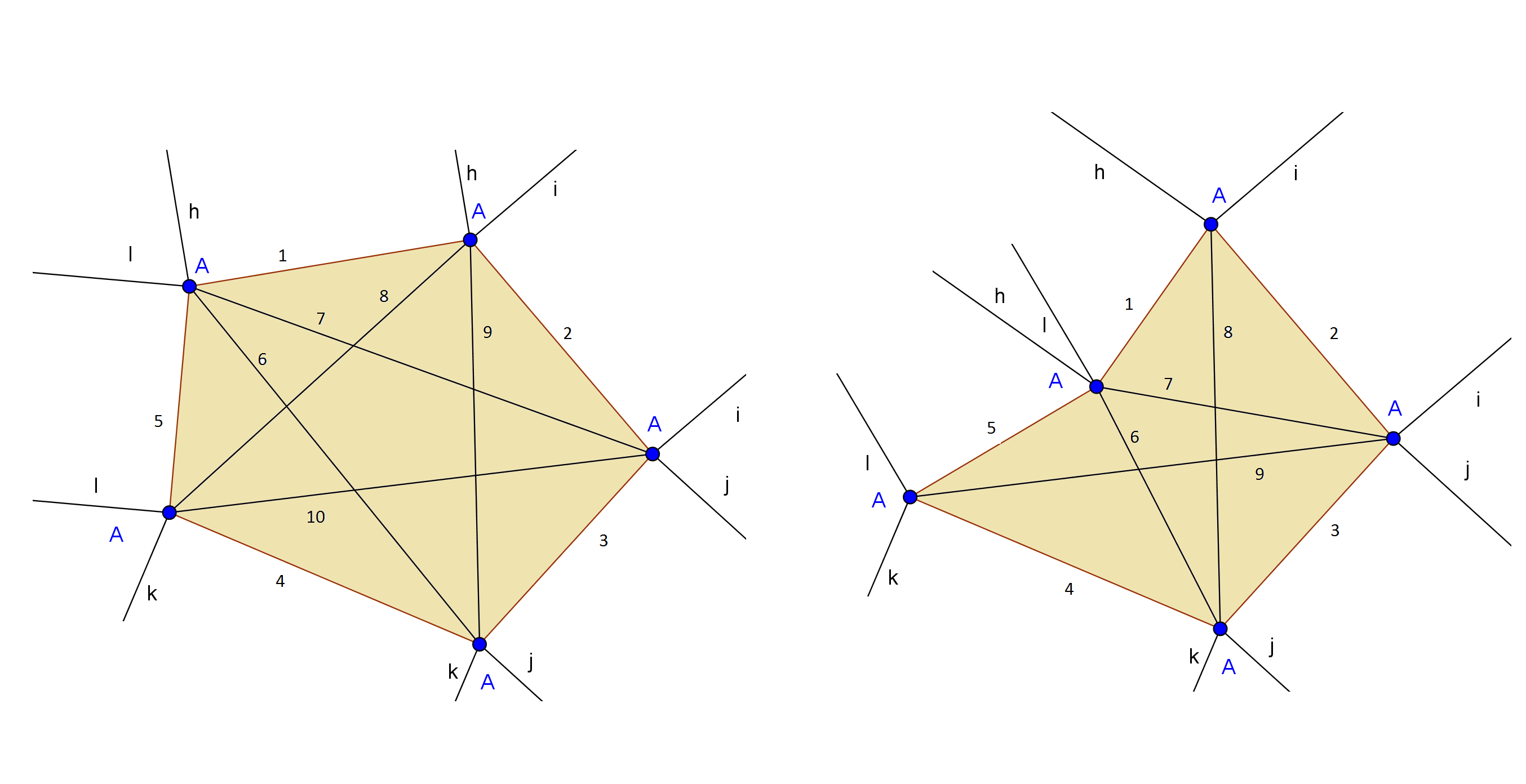}
\caption{Two surfaces of $\mathcal{H}^{1}(3,-1^{5})$ : the first has 10 saddle connections while the second has 9 saddle connections.}
\end{figure}

\subsection{Meromorphic differentials with prescribed residues}

If certain conditions are respected, every configuration of residues can be realized by a surface in a given stratum of $1$-forms. This lemma is crucial in the proof of Theorem 2.3.

\begin{lem}
Let $\mathcal{H}^{1}(a_1,\dots,a_n,-b_1,\dots,-b_m,-1^{s})$ be a stratum of meromorphic differentials with $s\geq0$ simple poles and $m\geq0$ poles of higher order on a topological surface of genus $g$. Let $(\mu_{1},\dots,\mu_{m},\nu_{1},\dots,\nu_{s})\in \mathbb{C}^{m} \times (\mathbb{C}^{\ast})^{s}$ be a family of numbers such that:\newline
(i) $\sum \limits_{i=1}^m \mu_{i} + \sum \limits_{j=1}^s \nu_{j} = 0$.\newline
(ii) Either $(\mu_{1},\dots,\mu_{m},\nu_{1},\dots,\nu_{s})$ generates $\mathbb{C}$ as a $\mathbb{R}$-vector space or there are exactly two nontrivial numbers among $(\mu_{1},\dots,\mu_{m},\nu_{1},\dots,\nu_{s})$.\newline
Then there is a meromorphic differential in the ambiant stratum such that the residues at the simple poles are $(\nu_{1},\dots,\nu_{s})$ and the residues at the poles of higher order are $(\mu_{1},\dots,\mu_{m})$.
\end{lem}

\begin{proof}
The case $m+s=1$ (one unique pole) is trivial and true. The case $m=0$ and $s=2$ (two simple poles and no poles of higher order) is also true because the corresponding strata are nonempty \cite{Bo} and the residues at simple poles are nontrivial. Rotating and rescaling provide every prescribed configuration of residues. In what follows, we exclude the previous cases. We first prove the lemma in the case $n=1$ and $g=0$.\newline
We consider a stratum $\mathcal{H}(a,-b_1,\dots,-b_m,-1^{s})$ of genus zero and a configuration of residues $(\mu_{1},\dots,\mu_{m},\nu_{1},\dots,\nu_{s})$. We are going to construct piece by piece a translation surface with poles with the adequate residues. We distinguish trivial and nontrivial residues.\newline
In Subsection 3.2, we construct neighborhoods of poles with nontrivial residues by removing the neighborhood of a semi-infinite line from an infinite cone. Thus, we get a translation surface (of infinite area) with a boundary that is a unique closed saddle connection whose holonomy corresponds to the residue at the pole. Analogously, we construct for every simple pole a semi-infinite cylinder whose boundary is a unique saddle connection whose holonomy corresponds to the residue at the pole. We call these surfaces \textit{nontrivial pieces}.\newline
For poles with trivial residues, we cut a slit starting from the conical singularity of the infinite cone so as to get a translation surface whose boundary is the union of two closed saddle connections whose holonomy are opposed to each other. These saddle connections can have any nontrivial holonomy. We call these surfaces \textit{trivial pieces}.\newline
Then, we have to connect every piece with the others as in a linear graph. By hypothesis, we have at least two nontrivial pieces. We chose one of them as one of the two terminal vertices. Then, we connect to this nontrivial piece every trivial piece in succession by pasting one of their boundary component with one of the boundary component of the other piece. This is always possible because we can rotate and rescale every trivial piece. In the end, we have a big translation surface with a boundary that is a unique saddle connection and $k$ nontrivial pieces remain. The boundary of each of these nontrivial pieces is a unique saddle connection whose holonomy vector is the residue of the pole. If there are exactly two nontrivial pieces, we are done. If there are at least three nontrivial pieces, we paste these pieces along the sides of a $k+1$-gon. In this way, the holonomy vectors of the sides are the holonomy vectors of the poles that belong to each piece. However, there are many different cyclic orders for the sides of this $k+1$-gon. Some of them generate a self-intersecting polygon. Among these cyclic orders, we choose the one defined by the cyclic order of arguments of the holonomy vectors of each side. Such a choice always generates a convex polygon. As $(\mu_{1},\dots,\mu_{m},\nu_{1},\dots,\nu_{s})$ generates $\mathbb{C}$ as a $\mathbb{R}$-vector space, this $k+1$-gon is not degenerate. Therefore, we have a translation surface with poles in the stratum $\mathcal{H}^{1}(a,-b_1,\dots,-b_m,-1^{s})$ with the prescribed residues.\newline

We then generalize the previous result for higher genus. Let $\mathcal{H}^{1}(a,-b_1,\dots,-b_m,-1^{s})$ be a stratum of genus $g$ with a unique conical singularity. For every configuration of residues satisfying the conditions of the lemma, there is a surface $X_{0}$ in $\mathcal{H}^{1}(a-2g,-b_1,\dots,-b_m,-1^{s})$ that realizes it. As we excluded the cases where there is one unique pole or two simple poles, we have $a-2g\geq1$ and a sufficient number of singularities ($n+p\geq3$). Then we start from the unique conical singularity in $X_{0}$ and we cut along $g$ slits. Then we paste cylinders inside each of them. This surgery is local so does not change the residues at the poles. Thus we get a surface of genus $g$ in $\mathcal{H}^{1}(a,-b_1,\dots,-b_m,-1^{s})$ with the prescribed residues. We have proved the lemma for every stratum with $n=1$.\newline

By now, we look for a surface with prescribed residues in $\mathcal{H}^{1}(a_1,\dots,a_n,-b_1,\dots,-b_m,-1^{s})$. The case already proved provides a surface $X$ in $\mathcal{H}^{1}(\sum \limits_{i=1}^n a_{i},-b_1,\dots,-b_m,-1^{s})$. We then break up the unique conical singularity to get a surface in the wanted strata. Breaking up is a local surgery that replaces one conical singarity by several conical singularities connected by small saddle connections, see \cite{Bo, EMZ}. This surgery does not modify the residues at the poles. This ends the proof.\newline
\end{proof}

There are examples of configurations of residues that cannot be realized in a givn stratum. They do not satisfy the hypothesis of Lemma 4.16.

\begin{ex}
In $\mathcal{H}^{1}(2,-2^{2})$, every differential from the stratum is of the form $\lambda\dfrac{z^{2}}{(z-1)^{2}}dz$ where $\lambda \in \mathbb{C}^{\ast}$. The residue of the differential at the double pole $z=1$ is $2\lambda$ while the residue at the double pole $z=\infty$ is $-2\lambda$. Any configuration $(u,-u)$ where $u \in \mathbb{C}^{\ast}$ satisfies the hypothesis. On the contrary, there is no meromorphic differential whose residues at the poles are both trivial.\newline
\end{ex}

\subsection{Quadratic differentials with prescribed periods}

For further applications, we give a systematic construction of flat surfaces such that the boundaries of domains of poles have prescribed holonomy vectors. Lemma 4.18 is a variant of Lemma 4.16. Breaking up conical singularities of quadratic differentials is not a local surgery. Therefore, it may modify the holonomy vectors of boundary edges of domains of poles (see Subsection 3.4). Consequently, the scope of Lemma 4.16 is narrower than its counterpart.

\begin{lem}
We consider surfaces of stratum $\mathcal{H}^{2}(a,-b_1,\dots,-b_p)$ of genus $g$ with labelled poles $P_{1},\dots,P_{p}$. Let $(\mu_{1},\dots,\mu_{p})\in \mathbb{C}^{p}$ with $p\geq2$ be a family of numbers such that:\newline
(i) $\sum \limits_{i=1}^p \mu_{i} = 0$.\newline
(ii) Either $(\mu_{1},\dots,\mu_{p})$ generates $\mathbb{C}$ as a $\mathbb{R}$-vector space or there are exactly two nontrivial numbers among $(\mu_{1},\dots,\mu_{p})$.\newline
Then there is a meromorphic quadratic differential in $\mathcal{H}^{2}(a,-b_1,\dots,-b_p)$ such that for every nonzero number $\mu_{i}$, the boundary of the domain of pole $P_{i}$ of order $b_{i}\geq3$ is a unique closed saddle connection whose length and direction (modulo $\pi$) are that of $\mu_{i}$. For poles of order two, $\mu_{i}$ is instead the holonomy vector of the waist curves of the cylinder.
\end{lem}

\begin{proof}
First, we prove the lemma in the case where $g=0$. We follow the same strategy as in the proof of Lemma 4.16. We consider a stratum $\mathcal{H}^{2}(a,-b_1,\dots,-b_m)$ of genus zero and a configuration of numbers $(\mu_{1},\dots,\mu_{p})$. We are going to construct a flat surface with the adequate saddle connections piece by piece.\newline
Just like in the proof of Lemma 4.16, \textit{nontrivial pieces} are domains of poles whose boundary is a unique closed saddle connection whose holonomy vector is $\mu_{i} \neq 0$.\newline
Likewise, \textit{trivial pieces} are domains of poles whose boundary is the union of two saddle connections with the same length and the same direction (modulo $\pi$). This construction is also permitted for poles of order two, see Figure 2.\newline

Pieces are connected with the others in the same way as in the proof of Lemma 4.16. Trivial pieces are connected as in a linear graph. By hypothesis, we have at least two nontrivial pieces. We chose one of them as one of the two terminal vertices. Then, we connect to this nontrivial piece every trivial piece in succession by pasting one of their boundary components with one of the boundary components of the other piece. This is always possible because we can rotate and rescale every trivial piece. In the end, we have a big flat surface with a boundary that is a unique saddle connection and $k$ nontrivial pieces remain. If there are exactly two nontrivial pieces, we are done. If there are at least three nontrivial pieces, we paste these pieces along the sides of a $k+1$-gon. Just like in the proof of Lemma 4.16, since $(\mu_{1},\dots,\mu_{p})$ generates $\mathbb{C}$ as a $\mathbb{R}$-vector space, there is a cyclic order of the sides such that the polygon is not degenerate nor self-intersecting.Therefore, we have a flat surface with poles of higher order in the stratum $\mathcal{H}^{2}(a,-b_1,\dots,-b_m)$ with the prescribed saddle connections.\newline

We then generalize the previous result for higher genus. Let $\mathcal{H}^{2}(a,-b_1,\dots,-b_p)$ be a stratum of genus $g$ with a unique conical singularity. For now, we exclude the case $p=2$ and $b_{1}=b_{2}=2$. For every configuration of numbers $(\mu_{1},\dots,\mu_{p})$ satisfying the conditions of the lemma, there is a surface $X_{0}$ in $\mathcal{H}^{2}(a-4g,-b_1,\dots,-b_p)$ that realizes it because the lemma is proved in genus zero. Then we start from the unique conical singularity in $X_{0}$ and we cut along $g$ slits internal to the central polygon. Then we paste cylinders inside each of them. This surgery is local so does not change the holonomy vectors of the saddle connections of the boundaries. Thus we get a surface of genus $g$ in $\mathcal{H}^{2}(a,-b_1,\dots,-b_p)$ with the prescribed holonomy vectors. In order to get a surface with prescribed holonomy vectors in stratum $\mathcal{H}^{2}(a,-2^{2})$ of genus $g$ with $g\geq1$, we cut $g$ slits with the same starting point in an infinite cylinder and add handles in the slits. We choose the width of the cylinder so as to realize every prescribed holonomy vector. This ends the proof.\newline
\end{proof}

\begin{rem}
Condition $\sum \limits_{i=1}^p \mu_{i} = 0$ is simply a technical hypothesis here to ensure the existence of a flat polygon essential for carrying out the proof. It does not reflect a deeper phenomenon. It is not a necessary condition.
\end{rem}

\section{Dynamics}

\subsection{Types of trajectories}

Using the canonical $k$-cover, every result about the dynamics of the vertical flow for simple meromorphic differentials can be transfered to meromorphic $k$-differentials. Therefore, in this section, most results are about translation surfaces with poles (corresponding to meromorphic $1$-forms).

Dynamics of flat surfaces with poles of higher order are different from the usual case because most trajectories go to infinity, that is to a pole. In the following, we essentially follow the definitions Strebel gives in \cite{St}.

\begin{defn} Depending on the direction, a trajectory starting from a regular point is of one of the four following types:\newline
- regular closed geodesic (the trajectory is periodic).\newline
- critical trajectory (the trajectory reaches a conical singularity in finite time).\newline
- trajectory converging to a pole of higher order.\newline
- recurrent trajectory (infinite trajectory nonconverging to a pole of higher order).
\end{defn}

\begin{prop} Recurrent trajectories lie in the core and for any recurrent trajectory there is a saddle connection in the same direction (up to a rotation of angle $\dfrac{2\pi}{k}$).
\end{prop}

\begin{proof} Lemma 3.1 states that every pole of higher order has a basis of neighborhoods whose complement is convex. So a trajectory that enters in one of these neighborhoods never comes back. Therefore, if a pole of higher order is in the closure of trajectory, then this trajectory finishes at this pole. So, every recurrent trajectory lies in a compact set. Consequently, the convex hull of a recurrent trajectory is a convex compact whose boundary has finite perimeter. As lengths of saddle connections are bounded below by a positive number, the boundary is a finite union of saddle connections following the direction of the recurrent trajectory. Therefore, it belongs to the core.
\end{proof}

\begin{prop} In translation surfaces with poles, regular closed geodesics follow saddle connection directions.
\end{prop}

\begin{proof} A closed geodesic describes a cylinder bounded by either a simple pole and a chain of colinear saddle connections or two chains of saddle connections. A cylinder is not bounded by two simple poles because it would be a Riemann sphere without marked points and two simple poles. Yet, we assumed that the set of conical singularities is nonempty. This is a contradiction.
\end{proof}

\begin{cor} Lebesgue-almost every trajectory that passes through a given point $x \in X$ finishes at a pole.
\end{cor}

\begin{proof} Recurrent trajectories and closed geodesics follow directions of saddle connections (Propositions 5.2 and 5.3). Directions of saddle connections are countable. Directions of trajectories passing through $x$ and finishing at a zero are countable too.\newline
\end{proof}

\subsection{Decomposition into invariant components}

For classical translation surfaces, invariant components for the vertical flow are cylinders and minimal components \cite{Ma}. Proposition 5.5 shows that the picture is a little more complicated for translation surfaces with poles.

\begin{prop} Let $(X,\omega)$ be a translation surface with poles ($\omega$ is a meromorphic $1$-form). Cutting along all saddle connections sharing a given direction $\theta$, we obtain finitely many connected components called \textit{invariant components}. 
There are four types of invariant components: \newline
- \textbf{finite volume cylinders} where the leaves are periodic with the same period, \newline
- \textbf{minimal components} of finite volume where the foliation is minimal and whose dynamics are given by a nontrivial interval exchange map, \newline
- \textbf{infinite volume cylinders} bounding a simple pole and where the leaves are periodic with the same period, \newline
- \textbf{free components} of infinite volume where generic leaves go from a pole to another or return to the same pole.\newline
\newline
Finite volume components belong to $core(X)$.\newline
The number of infinite volume cylinders is bounded by the number of simple poles.\newline
The number of free components is bounded by the total number of poles.
\end{prop}

\begin{proof} There are finitely many saddle connections in a given direction so we get finitely many connected components. Periodic geodesics describe cylinders bounded by either a simple pole and a chain of colinear saddle connections or two chains of saddle connections. In a finite volume component different from a cylinder, the generic leaf is recurrent in two directions. The closure of any recurrent leaf is bounded by vertical saddle connections. Therefore, the foliation is minimal in any invariant component that contains a recurrent leaf. Thus the remaining components have infinite volume, where leaves go from a pole to another (or possibly return to the same pole) or to a zero for finitely many leaves.\newline
There is at least one pole in each infinite volume component and there is a unique simple pole at the end of each infinite volume cylinder.
\end{proof}

\begin{ex} In $\mathcal{H}^{1}(1^{2},-1^{2})$, some surfaces have minimal components, see Figure 9.

\begin{figure}
\includegraphics[scale=0.3]{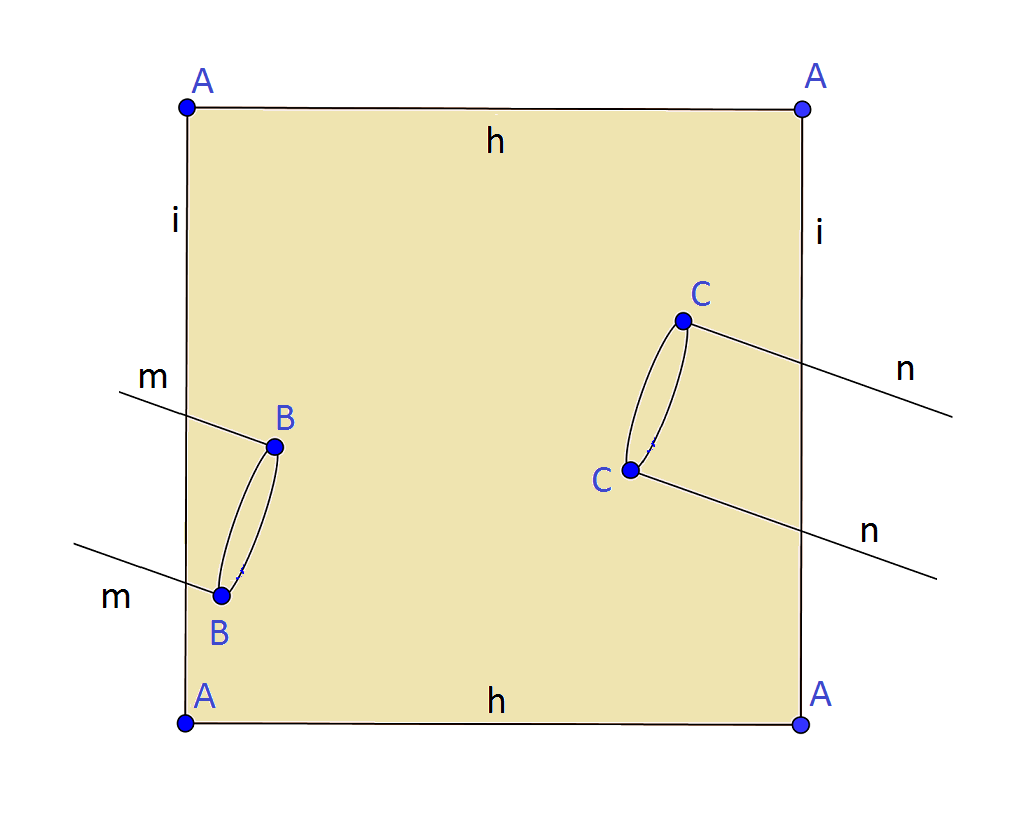}
\caption{This surface is obtained by pasting two infinite cylinders along a pair of slits following an irrational leaf of a flat torus}
\end{figure}

\end{ex}

Combinatorial considerations will give sharp bounds on the number of finite volume components but for now, we finish up with dynamical preliminary results.\newline

\subsection{Distribution of directions}

\begin{prop} For any translation surface with poles $(X,\omega)$ and $x \in X\setminus \Delta$, the set of directions $\theta$ such that the trajectory beginning in x along the $\theta$ direction finishes in a given pole $p_i$ is open (possibly empty).
\end{prop}

\begin{proof} A trajectory finishing at a pole stays at most a finite time if any in the core. Besides, it cannot return back to the core after leaving it. So $\exists \varepsilon > 0$ such that there is no conical singularity in a $\varepsilon$-tubular neighborhood of the trajectory. Therefore, we can slightly vary the direction while conserving a trajectory finishing at the same pole.
\end{proof}

\begin{cor}
The set of directions going to each pole is measurable so there is a function that assigns to every regular point a probability measure on the set of poles corresponding to the measure of the set of directions going to each pole.
\end{cor}

\begin{rem}
For a point of a translation surface with poles, the set of directions defining trajectories going to a given pole can be an infinite union of disjoints intervals, see Figure 10.
\end{rem}

\begin{figure}
\includegraphics[scale=0.3]{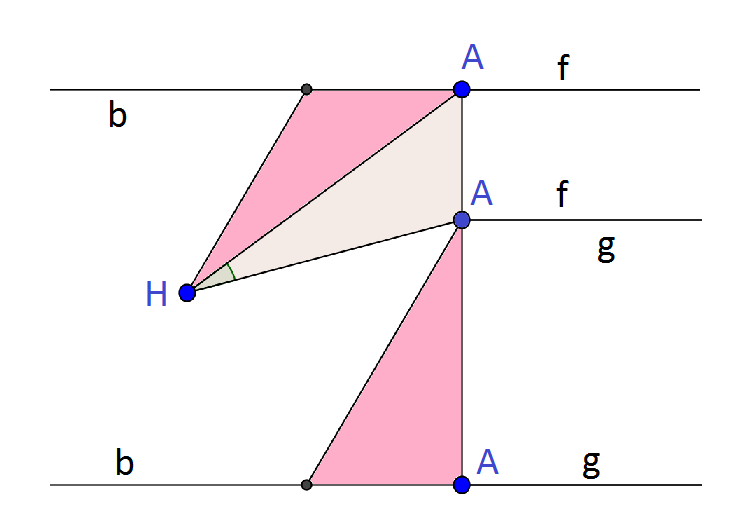}
\caption{In this surface in $\mathcal{H}^{1}(1,-1^{3})$, set of directions of trajectories going from a particular point to a given pole is an infinite union of disjoints intervals.}
\end{figure}

The set of directions of saddle connections can be deeply studied with the tools of descriptive set theory, see \cite{Au}. We shall remain on the surface and merely look at the accumulation points of this set in the unit circle. Proposition 5.10 shows that they correspond to finite volume cylinders and minimal components.

\begin{prop} Let $(X,\phi)$ be a flat surface with poles of higher order. Then the directions of the waist curves of finite volume cylinders and the directions of minimal components of $(X,\phi)$ are exactly the accumulation points of directions of saddle connections of $(X,\phi)$ considered as points of the unit circle (up to a rotation of angle $\dfrac{2\pi}{k}$). This holds for any $k$-differential with $k \geq 1$.
\end{prop}

\begin{proof} In a finite volume cylinder or in a minimal component along a given direction $\theta$, one can always find a sequence of saddle connections such that the sequence of their directions converges to $\theta$. So directions of finite volume cylinders and of minimal components are accumulation points of directions of saddle connections considered as points of the unit circle.\newline
If there are no finite volume cylinders nor minimal components in a given direction, without loss of generality, we suppose that the vertical direction is an accumulation direction of saddle connections. So, there is a conical singularity $x_{0}$ which is the origin of a family of saddle connections whose directions approch the vertical one. If we consider a small horizontal geodesic segment $I$ passing through $x_{0}$, every vertical trajectory starting from a point of this segment is a critical trajectory or a trajectory finishing at a pole, see Figure 11.\newline
\begin{figure}
\includegraphics[scale=0.3]{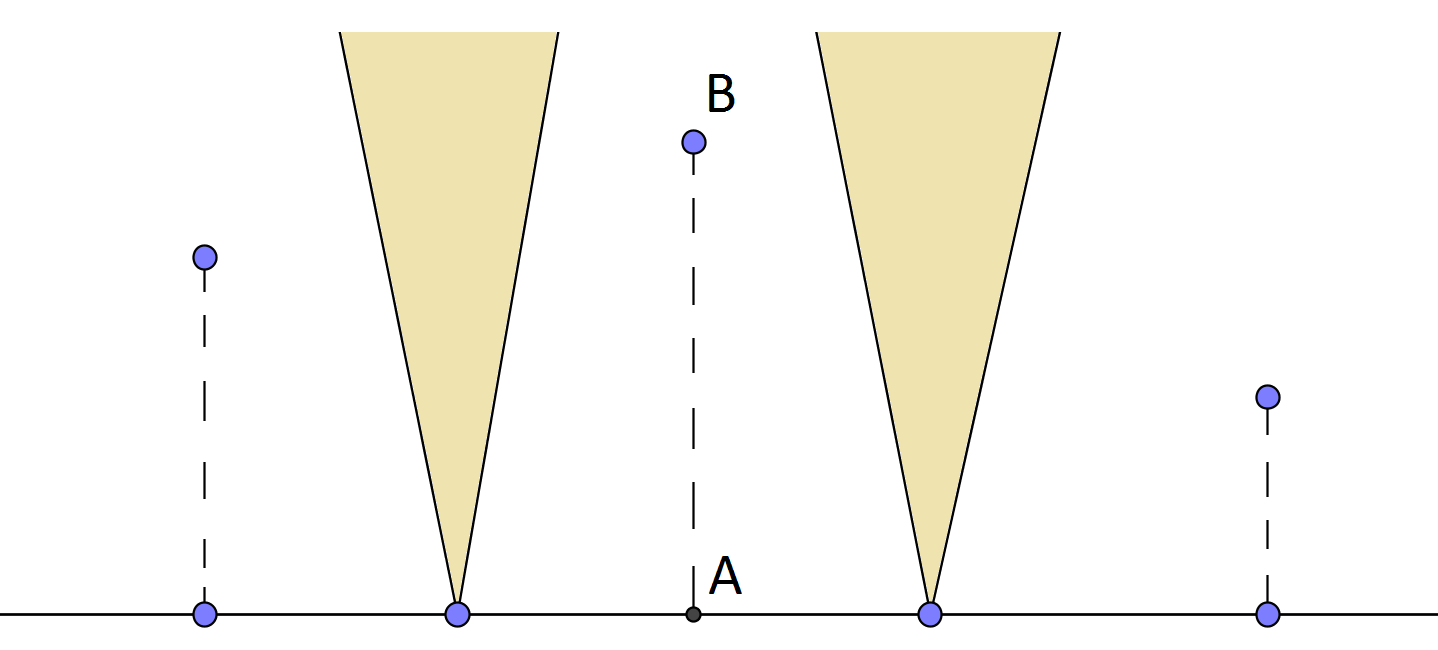}
\caption{Vertical trajectories starting from points of a horizontal geodesic segment go to a pole.}
\end{figure}
There is a finite number of critical trajectories in this segment because there is a finite number of zeroes and their opposite vertical trajectories leave the core after a finite time (there is a finite number of saddle connections in a given direction). So we can shorten our segment and get an open segment still containing $x_{0}$ such that all vertical trajectories starting from its points go to one and the same pole for points of the same side of the segment (left or right to $x_{0}$).\newline
Following Proposition 5.7, there is an open interval of directions around the vertical one such that trajectories go to this pole. So, a family of saddle connections approaching the vertical cannot exist.
\end{proof}

\begin{cor} Let $(X,\omega)$ be a translation surface with poles. The set of directions of saddle connections of $(X,\omega)$ in the unit circle is a closed proper subset.
\end{cor}

\begin{proof}
Following Propositions 5.2 and 5.3, directions of minimal components and of the waist curves of cylinders are directions of saddle connections. Therefore, the set of directions of saddle connections contains its accumulation points.\newline
\end{proof}

Using a theorem of Vorobets stated in \cite{Vo}, we are able to prove existence of a finite volume cylinder inside every minimal component. This will be crucial to relate infinity of the number of saddle connections and the combinatorics of the singularity pattern of the strata. A similar result has been proved by Aulicino in Lemma 4.4 of \cite{Au} in the context of slit translation surfaces with a quite different proof.

\begin{lem}
Let $X$ be a translation surface with or without poles. For a minimal component $\mathcal{U}$ in the vertical direction, there is a family of regular periodic geodesics that describe a finite volume cylinder inside $\mathcal{U}$.
\end{lem}

\begin{proof}
As the boundary of $\mathcal{U}$ is made of vertical saddle connections whose total holonomy is zero, we can glue the boundary components on each other while adding marked points if necessary or modifying the order of the conical singularities. We get a translation surface $Y$ without poles but with some marked vertical saddle connections (those that have been glued on each other). The translation surface $Y$ belongs to a stratum $\mathcal{H}$.\newline
Following Theorem 1.3 of Vorobets in \cite{Vo}, there are two constants $A(\mathcal{H})$ and $L(\mathcal{H})$ such that in every surface $M$ of $\mathcal{H}$, there is a cylinder whose area is at least $A(\mathcal{H})$ and whose length of waist curves is at most $L(\mathcal{H})$.\newline
We apply the Teichmüller flow in the vertical direction to $Y$ until the total length of the marked vertical saddle connections is smaller than $\dfrac{A(\mathcal{H})}{L(\mathcal{H})}$. We get a surface we denote by $Y^{1}$. The total length of the marked vertical saddle connections in $Y^{1}$ is smaller than $\dfrac{A(\mathcal{H})}{L(\mathcal{H})}$ whereas there is a cylinder whose height is larger than $\dfrac{A(\mathcal{H})}{L(\mathcal{H})}$. Therefore, a portion of the waist curves of the cylinder does not cross the marked vertical saddle connections. Existence of such a regular closed geodesic is an invariant property for the Teichmüller flow so there is a finite volume cylinder in $Y$ whose waist curves do not cross the marked vertical saddle connections. As $Y$ has been obtained from $\mathcal{U}$ by gluing the marked vertical saddle connections, there is a finite volume cylinder inside $\mathcal{U}$.\newline
\end{proof}

The previous lemma implies a simple characterization of surfaces with infinitely many saddle connections.

\begin{cor} A flat surface with poles of higher order $(X,\phi)$ of a stratum $\mathcal{H}$ has infinitely many saddle connections if and only if there is a finite volume cylinder in $X$. Besides, the locus in the stratum $\mathcal{H}$ where $|SC|=+\infty$ is an open set. 
\end{cor}

\begin{proof}
Following Lemma 5.12, if there is a minimal component in some direction, there is a finite volume cylinder in the surface. Therefore, an infinite number of saddle connections implies existence of a finite volume cylinder. As there is an infinite number of saddle connections in every finite volume cylinder, the equivalence is proved.\newline
\end{proof}

\subsection{Contraction flow}

In the case $k=1\ or\ 2$, there is a construction of surfaces with a degenerate core in the closure of every $GL^{+}(2,\mathbb{R})$-orbit. Such a construction works essentially because translation surfaces with poles have infinite area. It must be noted there is no natural action of $GL^{+}(2,\mathbb{R})$ on strata of $k$-differentials for which $k\geq3$. Therefore, such a construction does not generalize to $k$-differentials for arbitrary $k$.\newline

\begin{defn} Let $\alpha$ and $\theta$ be two distinct directions. The contraction flow is the action of the semigroup of matrices $C^{t}_{\alpha,\theta}$ conjugated to $\begin{pmatrix} e^{-t} & 0 \\ 0 & 1 \end{pmatrix}$ such that $\alpha$ is the contracted direction and $\theta$ is the preserved direction.
\end{defn}

\begin{lem}
Let $(X,\phi)$ be a flat surface with poles of higher order such that $\phi$ is a $1$-form or a quadratic differential. Then, there is a surface $(X_{0},\phi_{0})$ in the $GL^{+}(2,\mathbb{R})$-orbit closure of $(X,\phi)$ inside the ambiant stratum such that all saddle connections of $(X_{0},\phi_{0})$ share the same direction and $core(X_{0})$ is degenerate.
\end{lem}

\begin{proof}
Let $(X,\omega)$ be a translation surface with poles. The set of directions of saddle connections of $(X,\omega)$ in the unit circle (Corollary 5.11) is a closed proper set. Let $\theta$ be a direction that is not a direction of saddle connection. As the lengths of saddle connections of $(X,\omega)$ are bounded below by a positive number, no saddle connection shrinks (we excluded directions of saddle connections for this purpose) and the orbit of $(X,\omega)$ by a contraction flow preserving the direction $\theta$ converges to a translation surface with poles $(X_{0},\omega_{0})$ in the ambient stratum.\newline
We consider geodesic foliation in direction $\theta$. It is is not the direction of a saddle connection (Proposition 5.2 and 5.3) so generic leaves go from a pole to another pole (or return to the same poles). They form strips separated by critical leaves. This decomposition characterizes the surface. We pass from a leaf to another following direction $\alpha$. As direction $\alpha$ is preserved by the contraction flow, the limit of the decomposition in strips in the orbit of $(X,\omega)$ is another decomposition corresponding to a genuine translation surface with poles $(X_{0},\omega_{0})$ in the ambient stratum.\newline
As the contraction flow reduces the area of the core exponentially, $core(X_{0})$ is degenerate and every saddle connection belongs to direction $\alpha$.\newline
This construction is compatible with the canonical cover of half-translation surfaces by translation surfaces. Therefore, the lemma also holds for meromorphic quadratic differentials with poles of higher order.\newline
\end{proof}

\begin{ex} In $\mathcal{H}^{1}(2,-1^{2})$, pushing a surface through the contraction flow along a generic direction, we get a surface with a degenerate core, see Figure 12.\newline
\begin{figure}
\includegraphics[scale=0.3]{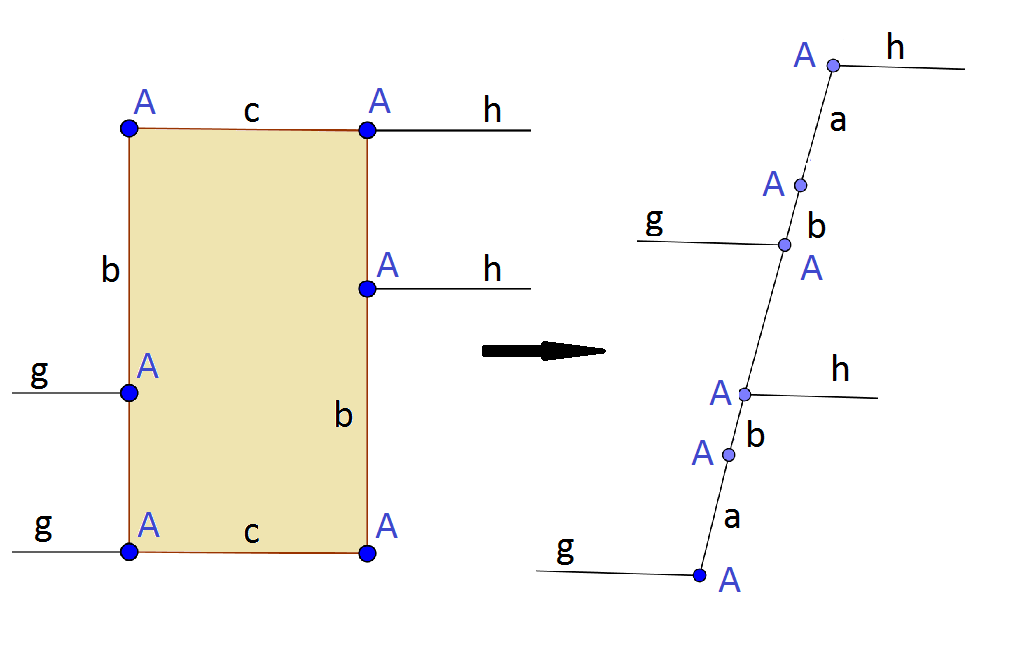}
\caption{A surface in $\mathcal{H}^{1}(2,-1^{2})$ and the surface obtained through the use of the contraction flow along a given direction}
\end{figure}
\end{ex}

\section{Reducibility index and invariant components}

Some complicated surfaces are constructed by connecting surfaces of genus zero with each other through cylinders. Such a possibility depends on the \textit{singularity pattern} of the stratum, that is the set with multiplicities encoding degrees of zeroes and poles of a stratum, see Figure 13.\newline
\textit{Graph representations} of the strata are the decomposition graphs allowed by the combinatorics of the singularity pattern, see Definition 2.1. In this section, we consider only the case of $1$-forms and quadratic differentials.\newline
\newline

\begin{figure}
\includegraphics[scale=0.3]{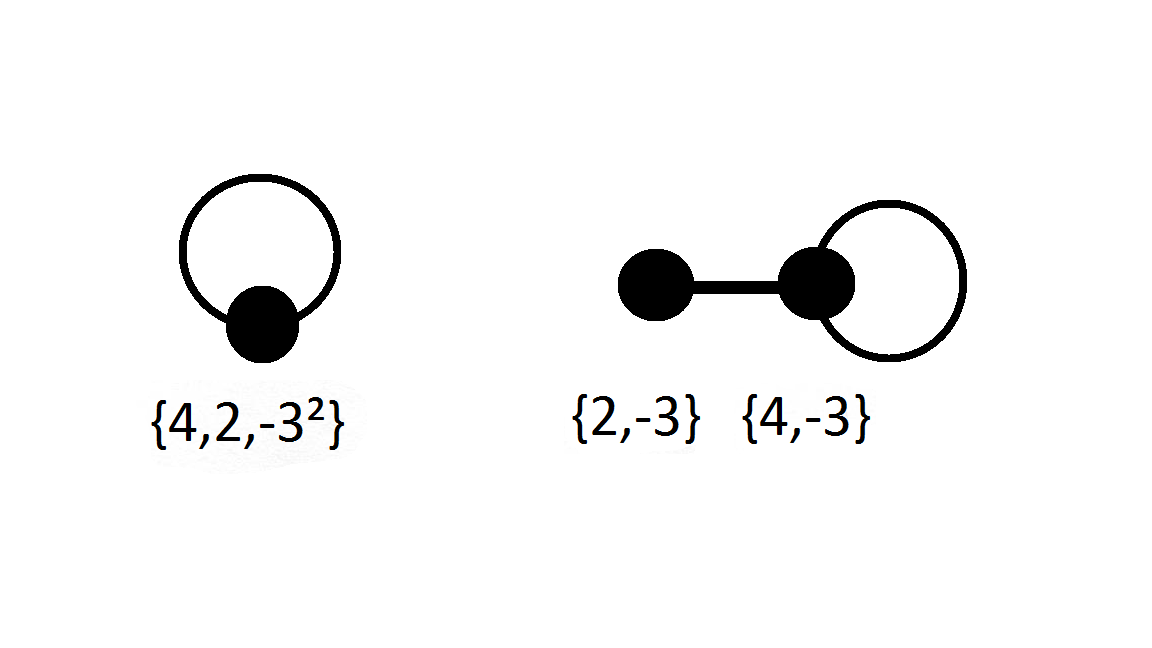}
\caption{The two pure graph representations of $\mathcal{H}^{1}(4,2,-3^{2})$.}
\end{figure}

Lemma 6.1 proves that cutting along cylinders actually realize graph representations that satisfy all the conditions.

\begin{lem}
Let $(X,\phi)$ be a surface in a stratum $\mathcal{H}=\mathcal{H}(a_1,\dots,a_n,-b_1,\dots,-b_p)$ and $F$ a family of finite volume cylinders of $X$ whose interiors are disjoints. Cutting along a waist curve in each cylinder of $F$ defines a graph representation $(G,f_0,\dots,f_s,w_0,\dots,w_s)$ of the stratum $\mathcal{H}$. Each vertex of graph $G$ corresponds to a connected component. Family $f_i$ is the set of singularities of $X$ that belong to the connected component corresponding to $i$ and $w_i$ is its genus. Each edge of graph $G$ is a cylinder of the family $F$. We say that $(X,F)$ realizes the graph representation $(G,f_0,\dots,f_s,w_0,\dots,w_s)$ of the stratum $\mathcal{H}$.\newline
\end{lem}

\begin{proof}
We prove that cutting along waist curves of cylinders of a family $F$ in a flat surface $(X,\phi)$ from a stratum $\mathcal{H}$, this defines a graph representation of $\mathcal{H}$.\newline
A finite volume cylinder is bounded at each end by a chain of saddle connections so there is a conical singularity in every connected component. If family $F$ is empty, we nevertheless have $n\geq1$ so there is a positive number in the singularity pattern of every connected component. Condition (i) is satisfied.\newline
The flat structure of each connected component is that of a flat surface whose singularities are those of the corresponding family, plus a pole of order $k$ for every adjacent cylinder that has been cut. Therefore, we have $\sigma(i)-k.v_i=k(2w_{i}-2)$ that is the formula of condition (ii).\newline
Condition (iii) is necessary only in the case $k=1$. If cutting an edge of the graph $G$ splits $G$ into two connected components, then cutting along a waist curve of the corresponding cylinder splits $X$ into two connected components.
If every pole belonged to the same connected component, the sum of the residues at the poles in this component would be zero and the holonomy vectors of waist curves of the cylinder would be zero. Indeed, holonomy vectors are given by periods of the differential. Therefore, condition (iii) is satisfied. This defines a graph representation of the stratum.\newline
\end{proof}

For every combinatorially admissible graph representation of a stratum, Theorem 2.3 shows there is a surface realizing it. The combinatorial obstruction is the only one. We split the proof in two parts depending if $k$ equals $1$ or $2$.

\begin{proof}[Proof of Theorem 2.3 in the case $k=1$]
Let $\mathcal{H}=\mathcal{H}^{1}(a_1,\dots,a_n,-b_1,\dots,-b_p)$ be a stratum of genus $g$ and $(G,f_0,\dots,f_s,w_0,\dots,w_s)$ be a graph representation of level $s$ of the stratum $\mathcal{H}$.\newline
For every vertex, we want to construct a translation surface with the same singularities as in family $f_i$ to which we add a number of simples poles equal to the valency $v_i$. Indeed, domains of simple poles are cylinders (like edges). The total order of the singularities of the surface will be $\sigma(i)-v_{i}=2w_{i}-2$. Therefore, it will be a translation surface of genus $w_i$.\newline
We must check that there is a sufficient number of singularities. In each family, there is at least one conical singularity so the only case we have to study is when $w_i=0$. If the valency is at least two, there are at least three singularities. If the valency is one, then the sum of the degrees of the family is $-1$ and there is another pole. Finally, if the valency is zero, then the sum of the degrees of the family is $-2$. Since by definition, the graph $G$ is connected, the original stratum corresponds to genus zero and there are at least three singularities.\newline
So, for every family $f_i$ of the graph representation, there is a genus $w_i$ stratum in which we can pick the adequate component for our decomposition. We will see later how we can get a surface with prescribed residues at the poles in the stratum corresponding to every vertex.\newline

We construct the required surface by gluing cylinders along the boundaries of the domains of the added simple poles that correspond to valency. If the residues of the poles (the original poles and the additional simple poles) satisfy some mild conditions, Lemma 4.16 provides a systematic construction of each surface corresponding to a vertex.\newline
We have to check that the residues of each pole of each vertex are compatible with each other. There are two conditions:\newline
- the sum of the residues at the two ends of each of the $t$ cylinders corresponding to edges is zero.\newline
- the sum of the residues at the poles of the same vertex (including the simple poles corresponding to edges) sums to zero.\newline
Thus, the problem is reduced to solving a linear system of $t+s+1=2t+1+\sum \limits_{i=0}^s w_i -g$ equations with real coefficients in $2t+p$ indeterminates. Here, $t$ is the number of edges and $p$ is the number of poles.\newline

We first study the specific case where $g=\sum \limits_{i=0}^s w_i$ and $p=1$. Then, we have $t=s$. The genus of the graph is zero so cutting an edge splits the graph into two connected components. As we have $p=1$, the graph is the trivial one (one vertex and no edges). This graph representation is realized by every surface in the stratum if we choose the family of cylinders to be the empty one.\newline

Outside this specific case, the system is underdetermined. We are going to show that a solution with nonzero values for every $t$-variable is always permitted.\newline
The solutions of the system form a linear subspace $S$ of $\mathbb{C}^{2t+p}$. For a given variable, either it is zero for every solution of $S$ or is nonzero outside a strict subspace of $S$. Therefore, in a generic solution of $S$, the only variables that have a null value are those that are zero for every solution of $S$. For every $t$-variable, we will construct specific solutions where this variable is nonzero in order to prove that all $t$-variables are nonzero for the generic solution.\newline

We distinguish three kinds of $t$-variables:\newline
(i) those associated to an edge connecting a vertex to itself (a loop)\newline
(ii) those associated to an edge connecting two distinct vertices and whose cutting splits the graph into two connected components\newline
(iii) those associated to an edge connecting two distinct vertices and whose cutting does not split the graph into two connected components.\newline

The value of a $t$-variable of the first kind has no impact on the total residue of its vertex so we can always choose a nontrivial value.\newline
For $t$-variables of the second kind, there is at least one pole in each component so a generic choice of residues of poles gives a nonzero value for the $t$-variable.\newline
In the third case, the edge associated to the $t$-variable belongs to a cycle. We can put to zero all variables excepted those corresponding to edges of the cycle and then pick nontrivial values for the latter.\newline
In conclusion, for every graph representation of any stratum, $t$-variables are nonzero for a generic solution of the equations defined by the graph.\newline

We then show that we can always construct surfaces with prescribed residues for strata corresponding to vertices.\newline
There is a projection of vector space $S$ of solutions to the space of $p$-variables and $t$-variables of every vertex. For a given vertex, if the dimension of the image of the projection is at least two, then a generic solution in $S$ assigns to $p$-variables and $t$-variables of this vertex values that generate $\mathbb{C}$ as a $\mathbb{R}$-vector space. Therefore, we can use Lemma 4.16 in order to get a surface with prescribed residues at the poles in the stratum.\newline
Then, we study the case where the image of the projection is one-dimensional. Removing the vertex splits the graph into $l$ connected components. We have $l=1$ because otherwise there would have been two independent parameters to prescribe residues to the surface corresponding to the vertex. $t$-variables corresponding to the connection of the vertex with connected components satisfy a unique relation. Therefore, the vertex is connected to its complement by a unique simple pole or a pair of simple poles whose residues are opposed to each other (otherwise, there would be more than one parameter). In conclusion, there are two cases:\newline
- there is no $p$-variable that is generically nontrivial and the vertex is connected to its complement by a pair of simple poles ($t$-variables) whose residue are opposed.\newline
- there is a unique $p$-variable generically nontrivial and the vertex is connected to its complement by a unique simple pole.\newline
These two cases satisfy the hypothesis of Lemma 4.16.\newline

In conclusion, for a generic solution in $S$, we can construct surfaces corresponding to every vertex. Therefore, for every graph representation of any stratum we can construct a flat surface that realizes it.\newline

In order to get parallel cylinders, we choose a realization of the graph representation. We pick a direction $\theta$ that is not the direction of a saddle connection and another direction $\alpha$. We apply the contraction flow contracting direction $\theta$ and preserving direction $\alpha$ to every component of the surface (we exclude cylinders), see Definition 5.11 and Lemma 2.2. For every component, we get a surface whose residues at the poles belong to direction $\alpha$. Then, we paste cylinders to connect components. Thus, we obtain a translation surface with poles realizing the graph representation and cylinders corresponding to edges are parallel (in the direction $\alpha$). This ends the proof.\newline
\end{proof}

\begin{ex}
The singularity pattern of $\mathcal{H}(1^{4},-1^{4})$ can be splitted into four families that can be arranged to form a cyclic graph of length 4, see Figure 14.

\begin{figure}
\includegraphics[scale=0.3]{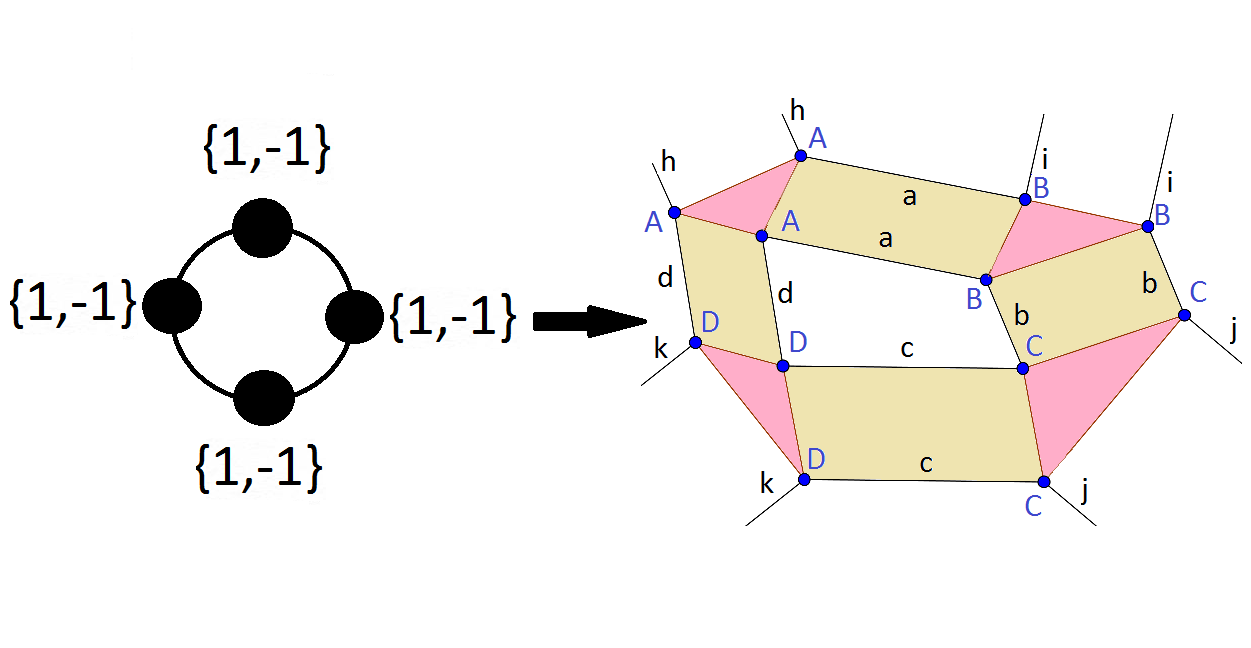}
\caption{A graph representation of a surface in $\mathcal{H}^{1}(1^{4},-1^{4})$ and a flat realization.}
\end{figure}

\end{ex}

\begin{rem}
Realization of graph representations associates a combinatorial object to a translation surface with or without poles by cutting along waist curves of the distinguished family of finite volume cylinders. If we just pinch waist curves to nodal points, we get meromorphic 1-forms on stable curves whose dual graph is the graph we would have obtained in the graph representation. Differentials on stable curves are a kind of half-way stage between combinatorial and geometrical objects. See \cite{Zv} for differentials on stable curves.
\end{rem}

\begin{proof}[Proof of Theorem 2.3 in the case $k=2$]
First we simplify the problem by considering only graph representations such that there is no edge cutting the graph into two connected components. Indeed, if each of the two parts of the graph representation can be realized, we can rescale and rotate one of the two surfaces in order to connect it with the other (waist curves of the cylinder corresponding to the additional poles of order two must have the same length and direction).\newline

We carry out another simplification. We want to consider only graph representations where there is a unique number $\geq-1$ in each family. In a given representation graph, we consider a family with several numbers $a_{1},\dots,a_{s}\geq-1$ such that their total sum $A$ is nonzero. In the singularity pattern of the stratum and the families of the representation graph, we replace the numbers $a_{1},\dots,a_{s}$ by the unique number $A$. If this simplified representation graph can be realized by a surface $X$, we can break up the conical singularity corresponding to the number $A$ in order to get a surface $X_{0}$ in the initial stratum that realizes the initial graph representation. Though breaking up singularities is not a local surgery, if the surgery is small enough, cylinders of $X$ corresponding to the edges of the graph representation persist.\newline

In the case of a family $(a_{1},\dots,a_{l},-b_{1},\dots,-b_{t})$ with several numbers corresponding to conical singularities and such that $A=\sum \limits_{i=1}^l a_{i} = 0$, we have to work differently. We have $\sum \limits_{i=1}^l a_{i} - \sum \limits_{j=1}^t b_{j}=-4+4w+2v$ where $w$ and $v$ are respectively the weight and the valency of the vertex to which the family is assigned. Indeed, weight $w$ corresponds to the genus of the connected component and edges are accounted for poles of order two because their geometry is that of a semi-infinite cylinder.\newline
We excluded the case $v=1$ (because the corresponding edge would cut out the graph into two connected parts) and the case $v=0$ is trivial. Thus we have $v\geq2$. Since $A=0$, the equation becomes $-\sum \limits_{j=1}^t b_{j}=-4+4w+2v$. Therefore, we have $w=0$, $v=2$ and there are no numbers below $-1$. The family is $(a_{1},\dots,a_{q},-1^{r})$ with $\sum \limits_{i=1}^q a_{i} = r$. We replace this family by $(2,-2)$. We will carefully look at this pole of order two.\newline
If the new graph representation we get can be realized by a surface $X$, we can slightly deform the surface in such a way that the boundary of the pole of order two is a unique closed saddle connection connected to the conical singularity of order two. Then we fill this pole by replacing the boundary saddle connection by two half saddle connections connected by a conical singularity of order $-1$. Thus, we get the realization of a graph representation where the corresponding family is not $(2,-2)$ but $(1,-1)$ (as valency is nonzero it is not an empty stratum). Next, we break up the conical singularity of order one as in the previous paragraph.\newline

Henceforth, we consider only graph representations where there is a unique number larger than $-1$ (or equal) in every family and such that there is no edge splitting the graph into two connected components.\newline
Adding a number of poles of order two equal to the valency $v_i$ of the vertex to every family $f_i$ gives the singularity pattern of a flat surface (with poles of higher order) of genus $w_i$. When there is at least a pole of higher order in the singularity pattern, the stratum is nonempty (Proposition 3.2). We check that there is a sufficient number of singularities by carrying out the same proof as in the case $k=1$.\newline
Consequently, for every family $f_i$ of the graph representation, there is a genus $w_i$ stratum in which we can pick the adequate component. Following our construction in the proof of Lemma 4.18, for every such singularity pattern, we can construct a surface such that each domain of pole is unique closed saddle connection of prescribed length and direction. We will see later how to match these components.\newline

We construct the required surface by gluing cylinders along the boundaries of the domains of the additional poles of order two that correspond to valency. We assume two technical conditions:\newline
- the sum of the holonomy vectors at the two ends of each of the $t$ cylinders corresponding to edges is zero.\newline
- the sum of the holonomy vectors of the boundaries of domains of poles of the same vertex (including the poles of order two corresponding to edges) sums to zero.\newline
Thus, the problem is reduced to solving a linear system of $t+s+1=2t+1+\sum \limits_{i=0}^s w_i -g$ equations in $2t+p$ indeterminates.\newline

We excluded graphs where cutting some edge splits the graph into two connected components. So we can exclude the case where $g=\sum \limits_{i=0}^s w_i$ and $p=1$ where the genus of the graph is zero.\newline

Outside this specific case, the system is underdetermined. Exactly in the same way as in the proof of the case $k=1$, we can prove that generic solutions have nonzero values for every $t$-variable.\newline

In the same way as in the case $k=1$, we show that we can always construct surfaces with prescribed holonomy vectors for strata corresponding to vertices. We distinguish two cases depending on the projection of vector space $S$ of solutions to the space of $p$-variables and $t$-variables of every vertex. These two cases correspond to the two cases of Lemma 4.18.\newline
In conclusion, for a generic solution in $S$, we can construct surfaces corresponding to every vertex. Therefore, for every graph representation of any stratum we can construct the flat surface that realizes it.\newline

Using the same arguments as in the proof of the case $k=1$ (applying the contraction flow to the cylinders only), we can choose a realization of the graph representation where the waist curves of the cylinders belong to the same direction. This ends the proof.\newline
\end{proof}

\begin{ex}
The singularity pattern of $\mathcal{H}^{2}(2,1^{3},-1^{3},-2^{3})$ can be split into four families $(1,-1,-2)$, $(1,-1,-2)$, $(1,-1,-2)$ and $(2)$ that can be arranged to form a star tree with three leaves, see Figure 15.

\begin{figure}
\includegraphics[scale=0.3]{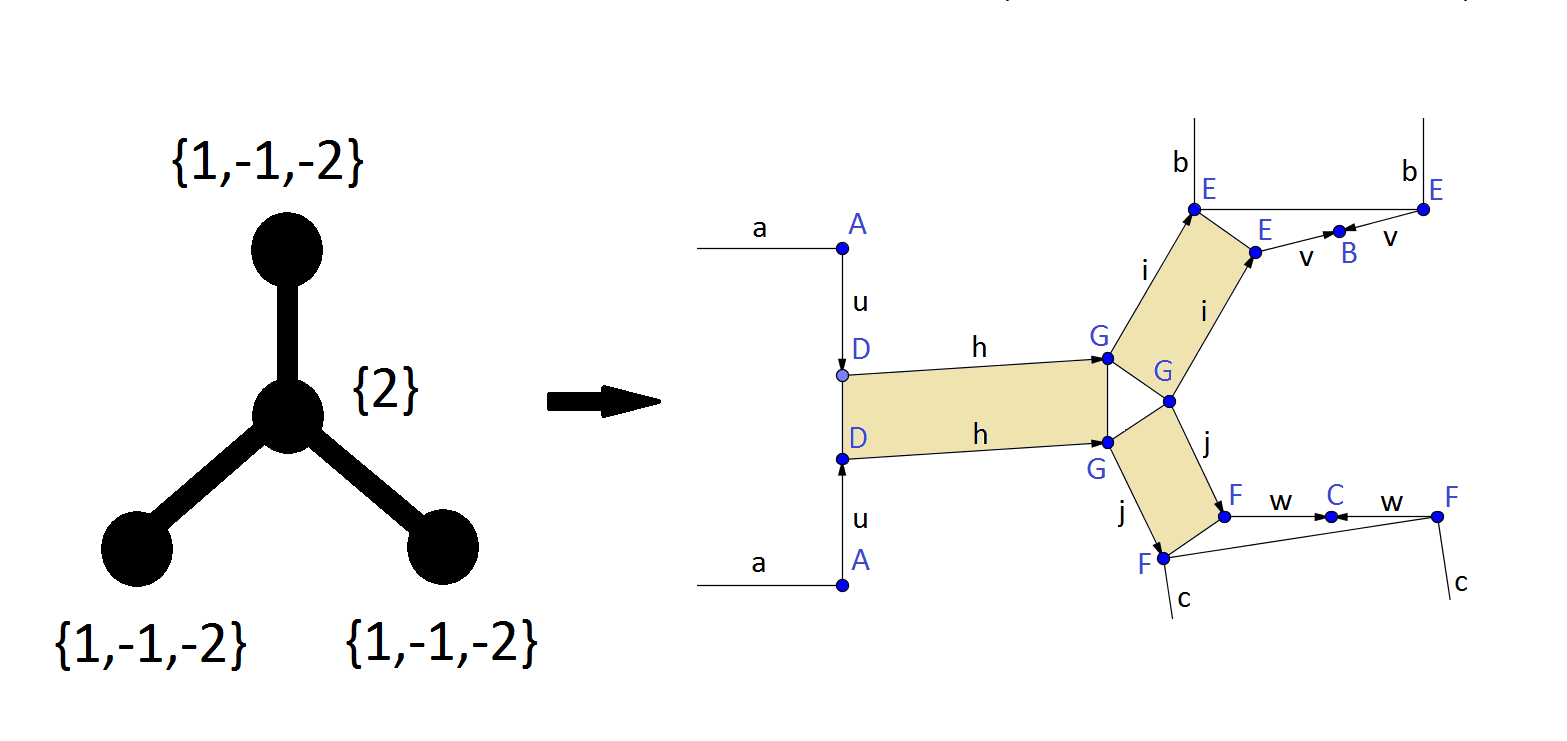}
\caption{A graph representation of a surface in $\mathcal{H}^{2}(2,1^{3},-1^{3},-2^{3})$ and a flat realization.}
\end{figure}

\end{ex}

We define for every stratum of $1$-forms or quadratic differentials its \textit{reducibility index} $\kappa$. The number $\kappa+1$ is an upper bound on the number of component the flat surface can decompose into after cutting along family of cylinders. This bound is an obstruction that depends only on the combinatorics of singularity pattern $a_1,\dots,a_n,-b_1,\dots,-b_p$ of stratum $\mathcal{H}^{k}(a_1,\dots,a_n,-b_1,\dots,-b_p)$.

\begin{defn}
The \textit{reducibility index} $\kappa$ of a stratum is the maximal number $s$ such that there exists a graph representation of level $s$ of the stratum. It follows from that definition that $0 \leq \kappa \leq n-1$.\newline
Strata of genus zero where $\kappa=0$ are called \textit{irreducible strata}. Otherwise, they are \textit{reducible strata}.\newline
\end{defn}

The combinatorial constraint $\kappa \leq n-1$ is not the only one. In the case $k=1$, there is another one that forces $\kappa$ to be small when both the genus and the number of poles are small even if there are many conical singularities.

\begin{prop}
Let $\mathcal{H}^{1}(a_1,\dots,a_n,-b_1,\dots,-b_p)$ be a stratum of genus $g$. If $g=0$ and $p=1$ then $\kappa=0$. Otherwise, we have $\kappa \leq 2g+2p-3$.
\end{prop}

\begin{proof}
We consider a pure graph representation of level $\kappa$ of stratum $\mathcal{H}$. In any graph, the valency sum formula states that the sum of the valency of each vertex is twice the number of edges.\newline
Set $c_{i}$ the number of vertices of valency equal to $i$. The valency sum formula states that $\sum \limits_{i\geq0} ic_{i} = 2(g+\kappa)$.\newline
Moreover, we have $\sum \limits_{i\geq0} ic_{i} \geq 3\sum \limits_{i\geq0} c_{i} -3c_{0} -2c_{1} - c_{2} \geq 3(\kappa+1) -3c_{0} -2c_{1} - c_{2}$.\newline
Thus, $\kappa \leq 2g+3c_{0}+(c_{1}+c_{2})+c_{1}-3$.\newline
There is a positive number in every family corresponding to a vertex in the graph decomposition so the valency of a vertex corresponding to a family without negative number is at least three. Consequently we have $c_{1} \leq p$ and $c_{1}+c_{2} \leq p$.\newline
Then we have $\kappa \leq 2g+2p+3c_{0}-3$.\newline
Except in the case where $\kappa=0$ and $g=0$ there is no vertex without edge so $c_{0}=0$ when $g \geq 1$ or $\kappa \geq 1$.\newline
When $g=0$ either $\kappa=0$ or $\kappa \leq 2g+2p-3$ so $\kappa=0$ when $p=1$. This ends the proof.
\end{proof}

Moreover, in genus zero, there are classical arithmetic obstructions. They show that the bound based on Proposition 6.6 is very far from being sharp.

\begin{prop}
Let $\mathcal{H}^{1}(a_1,\dots,a_n,-b_1,\dots,-b_p)$ be a stratum of genus zero such that 
$gcd(a_1,\dots,a_n,b_1,\dots,b_p) \neq 1$. Then the stratum is irreducible ($\kappa =0$).
\end{prop}

\begin{proof}
In genus zero, a graph representation is a graph with no cycles. So there always exists a vertex with valency $1$. We cannot find a subfamily of $a_1,\dots,a_n,-b_1,\dots,-b_p$ such that the sum of the degrees is $-1$. Thus, $\kappa = 0$.
\end{proof}

The following proposition reduces the calculation of the reducibility index $\kappa$ of the strata of higher genus to those of genus zero.

\begin{prop}
Let $\mathcal{H}=\mathcal{H}^{k}(a_1,\dots,a_n,-b_1,\dots,-b_p)$ be a stratum of genus $g$. We have $\kappa(\mathcal{H})= \kappa(\mathcal{H'})$ where $\mathcal{H'}=\mathcal{H}^{k}(a_1,\dots,a_n,-b_1,\dots,-b_p,-k^{2g})$ is the stratum of genus zero associated to $\mathcal{H}$ by adding a sufficient number of poles of order $k$ to obtain genus zero.
\end{prop}

\begin{proof}
We consider a graph representation of level $\kappa(\mathcal{H})$ of stratum $\mathcal{H}$ of genus $g$. Without loss of generality, we can choose a pure graph representation (see Definition 2.1). The graph of the graph representation has $g+\kappa$ edges. We remove $g$ edges in such a way the graph remains connected. We add to every vertex one pole of order $k$ per adjacent edge that has been removed.\newline
In the case $k=2$, we get a graph representation of level $\kappa(\mathcal{H})$ of stratum $\mathcal{H'}$. In the case $k=1$, we have to check another condition. In the new graph we get, if an edge becomes such that cutting it splits the graph into two connected components, it is because an edge connecting these two components has been cut. Therefore, in the case $k=1$, there is at least one simple pole in each of these connected components. Thus, we get a graph representation of level $\kappa(\mathcal{H})$ of stratum $\mathcal{H'}$. Therefore, we have $\kappa(\mathcal{H}) \leq \kappa(\mathcal{H'})$.\newline
Next, we consider a graph representation of level $\kappa(\mathcal{H'})$ of stratum $\mathcal{H'}$ of genus zero. We pick randomly $g$ pairs of poles of order $k$ among the families and remove them. For every pair removed, we draw an edge between their vertices (not necessarily distinct). We have constructed a graph representation of level $\kappa(\mathcal{H'})$ of stratum $\mathcal{H}$ of genus $g$.\newline
Therefore, we have $\kappa(\mathcal{H}) \geq \kappa(\mathcal{H'})$. This ends the proof.\newline
\end{proof}

In some specific cases, we are able to provide the value of the reducibility index of a stratum immediately.

\begin{cor}
Let $\mathcal{H}^{1}(a_1,\dots,a_n,-1^{p})$ be a stratum of translation surfaces without poles (when $p=0$) or whose poles are all simple (when $p\geq2$). Then we have $\kappa=n-1$.
\end{cor}

\begin{proof}
It suffices to check the case of strata $\mathcal{H}^{1}(a_1,\dots,a_n,-1^{p})$ of genus zero (Proposition 6.8). In this case, we have $p=2+\sum \limits_{i=1}^n a_i$. The case $n=1$ is trivial so we assume $n\geq2$. We form the following families $\lbrace a_{1},-1^{1+a_{1}}\rbrace$, $\lbrace a_{2},-1^{1+a_{2}}\rbrace$ and $\lbrace a_{i},-1^{a_{i}}\rbrace$ for $3\leq i \leq n$. We attribute to each family a vertex of a $A_n$ graph (linear graph with $n$ vertices) and thus obtain a graph representation of level $n-1$ of stratum $\mathcal{H}(a_1,\dots,a_n,-1^{p})$.\newline
\end{proof}

Theorem 2 of \cite{Na} proves that for classical translation surfaces, the maximal number of components invariant by the vertical flow is $g+n-1$. This bound is sharp. We are going to prove a similar result for (half-)translation surfaces with poles. In this context, the combinatorics of the strata are more complicated so we need the concept of reducibility index (see Definition 2.2) to quantify specific obstructions.\newline

\begin{thm}
Let $(X,\phi)$ be a flat surface in stratum $\mathcal{H}^{k}(a_1,\dots,a_n,-b_1,\dots,-b_p)$ of genus $g$. Cutting along all saddle connections sharing a given direction $\theta$, we obtain at most $g+\kappa$ finite volume components. The bound is sharp for every stratum.
\end{thm}

\begin{proof}
Let us suppose the existence of a flat surface $(X,\phi)$ with at least $g+\kappa+1$ finite volume components. These finite volume components are either minimal components or finite volume cylinders. Following Lemma 5.12, there is a finite volume cylinder inside every minimal component. So there are at least $g+\kappa+1$ finite volume cylinders whose interiors are disjoints in the surface $X$. Such a surface $(X,\phi)$ would realize a graph representation of level $\kappa+1$ and would contradict the hypothesis.\newline
For a stratum $\mathcal{H}=\mathcal{H}^{k}(a_1,\dots,a_n,-b_1,\dots,-b_p)$ of genus $g$ and of reducibility index $\kappa$, there is a graph representation of level $\kappa$. According to Theorem 2.3, there is a surface realizing this graph representation such that the $g+\kappa$ edges are represented by parallel cylinders so the bound is sharp.
\end{proof}

The reducibility index is crucial to characterize the strata where an infinity of saddle connections can occur.

\begin{proof}[Proof of Corollary 2.4]
Following Corollary 5.13, a surface has an infinite number of saddle connections if and only if there is a finite volume cylinder of periodic geodesics in some direction.\newline
If $g>0$ or $\kappa>0$, there is a graph representation of the surface with at least one edge so there is a surface with a cylinder in the stratum (Theorem 2.3).\newline
Conversely, in an irreducible stratum, the bound in Theorem 6.10 shows that there is no finite volume component so the number of saddle connections is always finite.
\end{proof}

We give some examples to illustrate the different phenomena that can occur.

\begin{ex} In $\mathcal{H}^{1}(3,-1^{3})$, we have $g=1$ and some surfaces have a cylinder so have an infinite number of saddle connections, see Figure 16.

\begin{figure}
\includegraphics[scale=0.3]{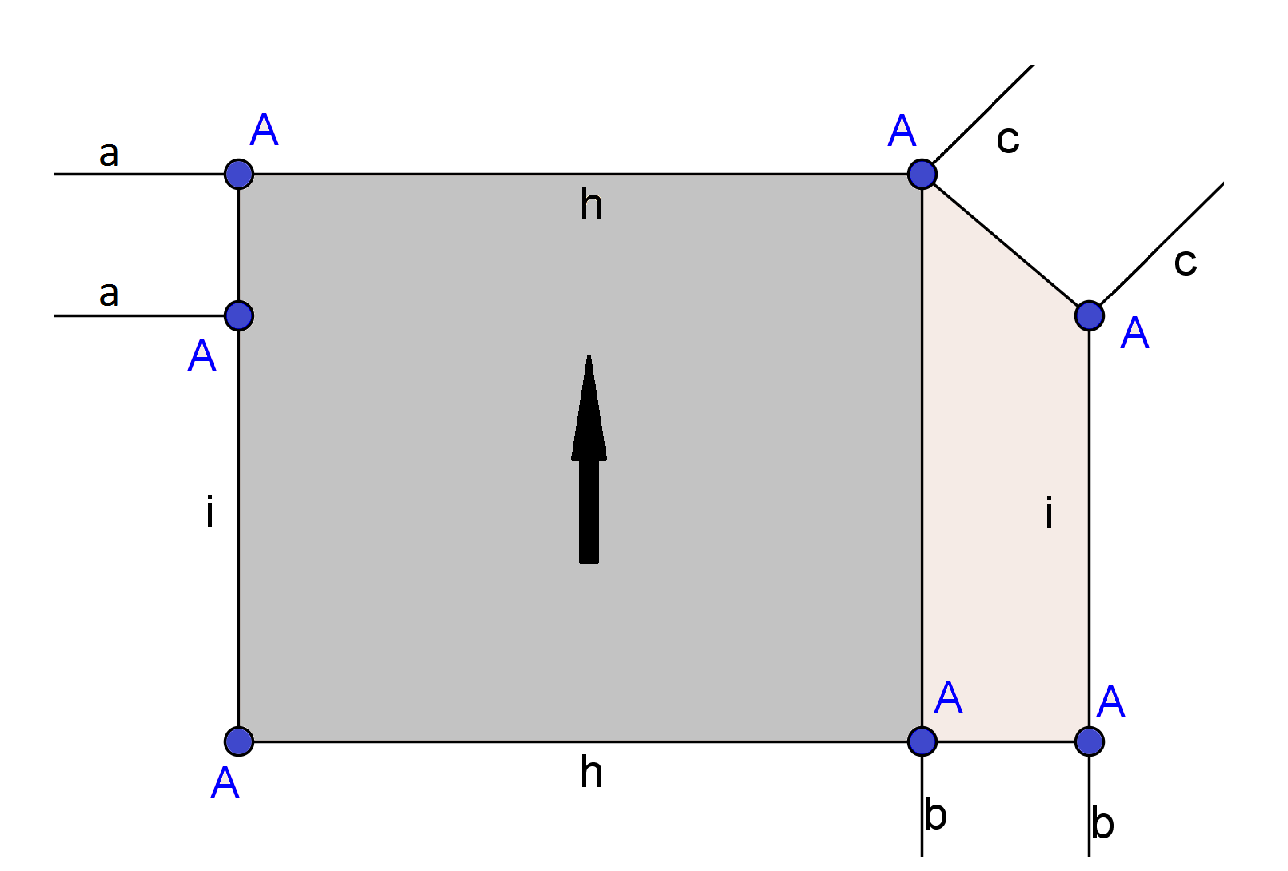}
\caption{A surface in $\mathcal{H}^{1}(3,-1^{3})$ with a cylinder in the vertical direction}
\end{figure}
\end{ex}
\begin{ex} In $\mathcal{H}^{1}(1^{2},-1^{4})$, we have $\kappa = 1$ and some surfaces have infinitely many saddle connections, see Figure 17.

\begin{figure}
\includegraphics[scale=0.3]{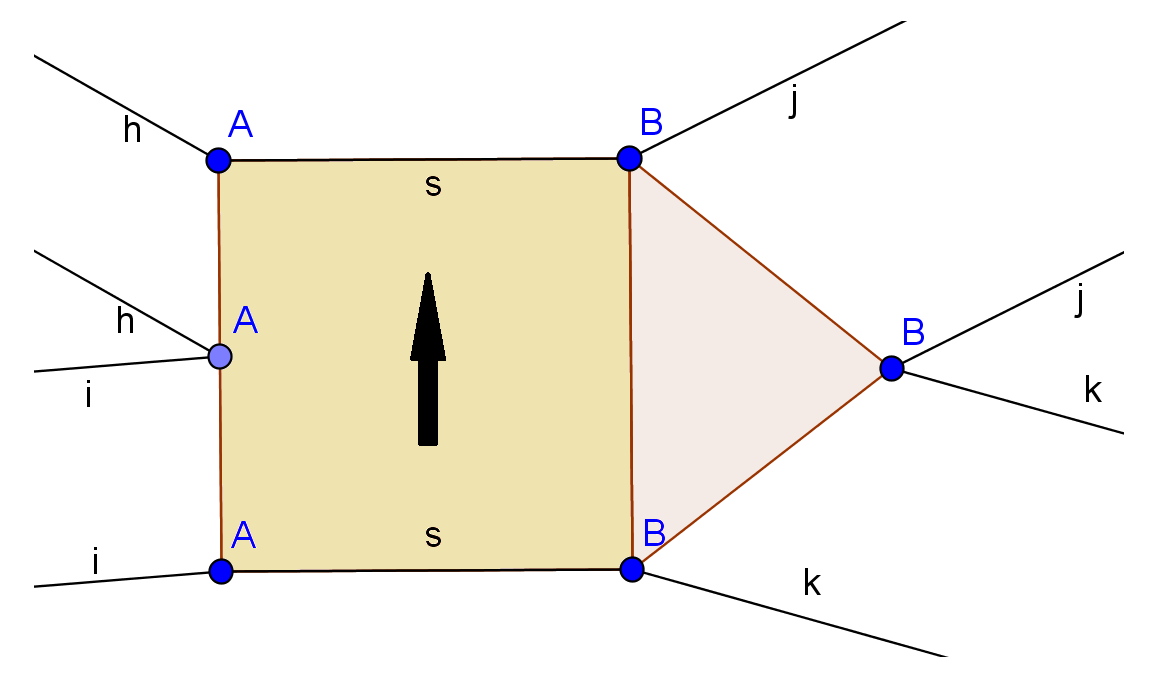}
\caption{A surface in $\mathcal{H}^{1}(1^{2},-1^{4})$ with a cylinder in the vertical direction}
\end{figure}
\end{ex}

\begin{ex} In $\mathcal{H}^{1}(1,7,-5^{2})$, we have $\kappa =0$ and all surfaces have finitely many saddle connections. An infinite number of saddle connections would imply existence of closed geodesic that would split the surface into two connected components where the total order of the singularities would be $-1$.\newline
\end{ex}

\section{Noncrossing saddle connections}

The following proposition provides lower bounds both on the number $|SC|$ of saddle connections and the maximal number $|A|$ of noncrossing saddle connections. In the case $k=1\ or\ 2$, the contraction flow is crucial to prove the sharpness of our lower bound on the number of saddle connections. In particular, this means that in the case $k=1\ or\ 2$, there are flat surfaces with a finite number of saddle connections in every stratum.

\begin{prop} Let $|A|$ be the number of edges of a maximal geodesic arc system for a flat surface with poles of higher order $(X,\phi)$ of genus $g$ belonging to $\mathcal{H}^{k}(a_1,\dots,a_n,-b_1,\dots,-b_p)$, then:
$$ |SC| \geq |A| \geq 2g-2+n+p$$
$|SC|=|A|=2g-2+n+p$ if and only if $core(X)$ is degenerate. Provided $k=1\ or\ 2$, there are such surfaces with degenerate core in the closure of every chamber in the ambiant stratum.
\end{prop}

\begin{proof}
We have $|A|=2g-2+p+n+t$ where $t$ is the number of ideal triangles in the triangulation defined by a maximal geodesic arc system (Lemma 4.8).
The contraction flow gives examples of surfaces that have a degenerate core (Lemma 5.15).\newline
When $core(X)$ is degenerate, $\partial\mathcal{C}(X)=core(X)$ so all saddle connections belong to $\partial\mathcal{C}(X)$. Therefore, $|SC|=2g-2+n+p$.\newline
\end{proof}

When $k\geq3$, there is no contraction flow so there is no general procedure to get examples of the equality case. Most importantly, in some strata, there is no flat surface whose core is degenerate.\newline

\begin{prop}
For any flat surface $(X,\phi)$ of $\mathcal{H}^{k}(a_1,\dots,a_n,-b_1,\dots,-b_p)$ such that $k\geq3$ and there is a pole whose order $b_j$ satisfies  $\dfrac{k}{2} < b_{j} < k$, $core(X)$ is not degenerate and $|SC|>2g-2+n+p$.\newline
\end{prop}

\begin{proof}
Poles of order $b<k$ are conical singularities of total angle equal to $\frac{k-b}{k}2\pi$. Following Proposition 4.11, interior angles of the boundary of a domain of pole are at least equal to $\pi$. Therefore, poles of order $\dfrac{k}{2} < b < k$ do not belong to the boundary of any domain of pole and the core cannot be degenerate.
\end{proof}

\begin{ex}
Every flat surface $(X,\phi) \in \mathcal{H}^{3}(4,-2^{2},-3^{2})$ is such that $core(X)$ is not degenerate, see Figure 18.

\begin{figure}
\includegraphics[scale=0.3]{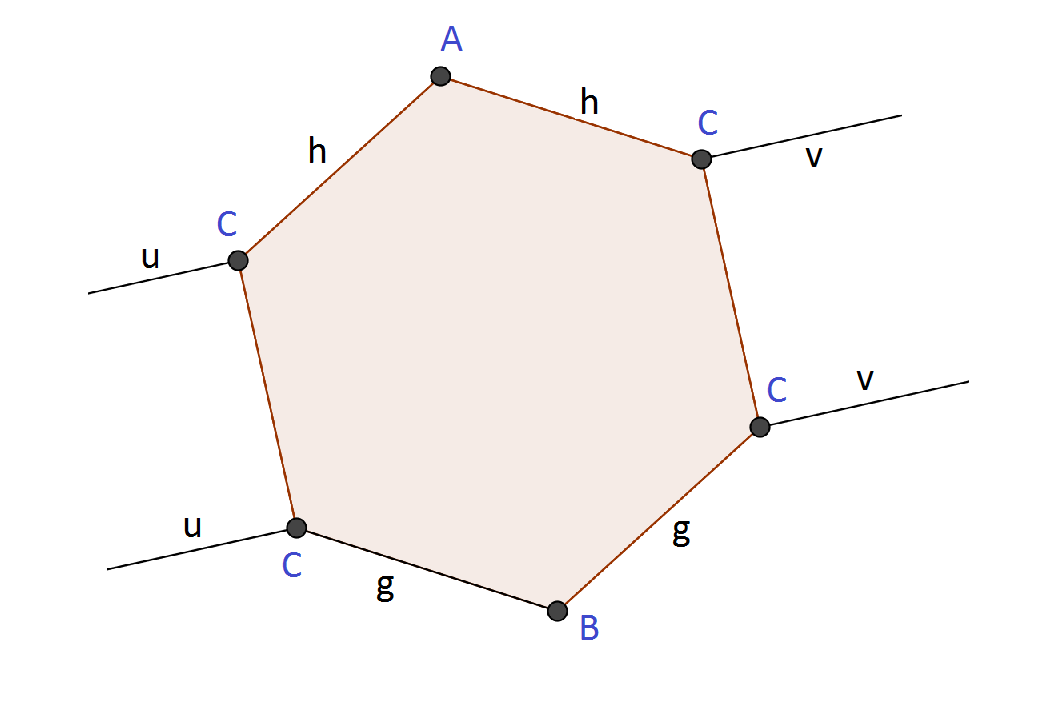}
\caption{A flat surface of $\mathcal{H}^{3}(4,-2^{2},-3^{2})$}
\end{figure}

\end{ex}

The following is a preparatory result to an upper bound depending on the combinatorics of the stratum.\newline

For a pole $P_j$ of order $b_j \geq k$, we denote by $d_j$ the number of saddle connections of the boundary of the domain $P_j$ whose two sides belong to the domain and by $e_j$ the number of saddle connections one side of which belongs to the domain of $P_j$ and another side belongs to another domain of pole or to the interior of the core. We define $\beta=2\sum \limits_{j=1}^p d_j + \sum \limits_{j=1}^p e_j$ as the boundary number of the surface.

\begin{lem}
Let $(X,\phi)$ be a flat surface of $\mathcal{H}^{k}(a_1,\dots,a_n,-b_1,\dots,-b_p)$. Set $|A|$ the number of edges of a maximal arc geodesic system and $\beta$ the boundary number of the surface. Then $|A| = 6g-6+3n+3p-\beta$.\newline
Besides, $|A|$ is constant in each chamber.
\end{lem}

\begin{proof}
We define $\phi_j$ as the sum of interior angles of the boundary of the domain of pole $P_j$.\newline
The degree of the Gauss map of a simple loop around a pole of order $b$ is $\dfrac{b-k}{k}$. This loop can be deformed in order to be as close as necessary to the boundary of the domain of the pole. The boundary of the domain of the pole is a broken line of saddle connections and therefore the sum of angular defects of the corners of the domain of the pole is $2\pi\dfrac{b-k}{k}$.\newline
The $d_j$ saddle connections whose two sides belong to the domain of the pole should be counted twice. Therefore, $\phi_j$ is given by the following formula:
$$ \phi_j = 2\pi d_j + \pi e_j + 2\pi\dfrac{b_{j}-k}{k}$$
Then, we get:
$$ \sum \limits_{j=1}^p \phi_j = 2\pi\sum \limits_{j=1}^p d_j + \pi\sum \limits_{j=1}^p e_j + 2\pi\sum \limits_{j=1}^p \dfrac{b_{j}-k}{k} $$

Let $t$ the number of ideal triangles in the triangulation of $core(X)$ given by the maximal geodesic arc system. The sum of the angles for $\mathcal{I}\mathcal{C}(X)$ is $\pi t$.\newline

The total sum of angles of conical singularities is $2\pi\sum \limits_{i=1}^n \dfrac{a_{i}+k}{k}$ so we have:
$$\pi t = 2\pi\sum \limits_{i=1}^n \dfrac{a_{i}+k}{k} - 2\pi\sum \limits_{j=1}^p d_j - \pi\sum \limits_{j=1}^p e_j - 2\pi\sum \limits_{j=1}^p \dfrac{b_{j}-k}{k}$$
Using $\sum \limits_{i=1}^n a_i - \sum \limits_{j=1}^p b_j = k(2g-2)$ and simplifying, we get:
$$ t = 4g-4+2n+2p - 2\sum \limits_{j=1}^p d_j - \sum \limits_{j=1}^p e_j$$
We have $|A|=2g-2+n+p+t$ (Lemma 4.10) so we get:
$$ |A| = 6g-6+3n+3p - 2\sum \limits_{j=1}^p d_j - \sum \limits_{j=1}^p e_j$$
The formula depends only on the shape of $\partial\mathcal{C}(X)$ so is an invariant of chambers.\newline
\end{proof}

As a corollary of Lemma 7.4, we finally get an upper bound on the number of edges of maximal geodesic arc systems. It is quite remarkable that the bound does not depend on the individual degrees of the zeroes and poles of the stratum.

\begin{cor} Let $|A|$ be the number of edges of a maximal geodesic arc system for a flat surface with poles of higher order $(X,\phi)$ of genus $g$ that belongs to $\mathcal{H}(a_1,\dots,a_n,-b_1,\dots,-b_p)$, then $|A| \leq 6g-6+3n+2p$.
\end{cor}

\begin{proof}
Every domain of pole is bounded by at least one saddle connection so we have $\beta \geq p$. Using Lemma 7.4, we get:
$$|A| \leq 6g-6+3n+3p-\beta \leq 6g-6+3n+2p$$
\end{proof}

The bound is not sharp. In many cases, it can be slightly improved. The following proposition provides some examples.

\begin{prop}
Let $(X,\omega)$ be a translation surface with poles of $\mathcal{H}^{1}(a_1,\dots,a_n,-b_1,\dots,-b_p)$. Let $|A|$ be the number of edges of a maximal arc geodesic system. We have:\newline
- $|A| \leq 6g+3n-5$ if $p=1$ and $g \geq 1$.\newline
- $|A| \leq 3n-6$ if $p=1$, $n \geq 3$ and $g=0$.\newline
- $|A| \leq 3n-3$ if $p=2$, $n \geq 2$ and the stratum is irreducible (in particular, $g=0$).
\end{prop}

\begin{proof}
In the specific case $p=1$, $|A|=6g+3n-4$ means $\beta=1$ so the boundary of the domain of the pole is a unique saddle connection. However, the residue at the pole is zero. This is a contradiction so we can slightly enhance the bound ($\beta \geq 2$).\newline
In the cases where there is no cylinder in the stratum (especially when strata are irreducible, see Definition 2.2), $core(X)$ cannot have two saddle connections as a boundary. So we have $\beta \geq 3$.\newline
\end{proof}

Following Theorem 7.1, the lower bound for $|A|$ and $|SC|$ may be attained only on very singular surfaces. Proposition 7.7 shows that staying away from the discriminant $\Sigma$ we may obtain different lower bounds for $|A|$. We consider the specific example of $k$-differentials whose poles of higher order are all of order $k$, see Figure 19 for an example. Generic differentials in such strata always display an optimal number of noncrossing saddle connections.

\begin{figure}
\includegraphics[scale=0.3]{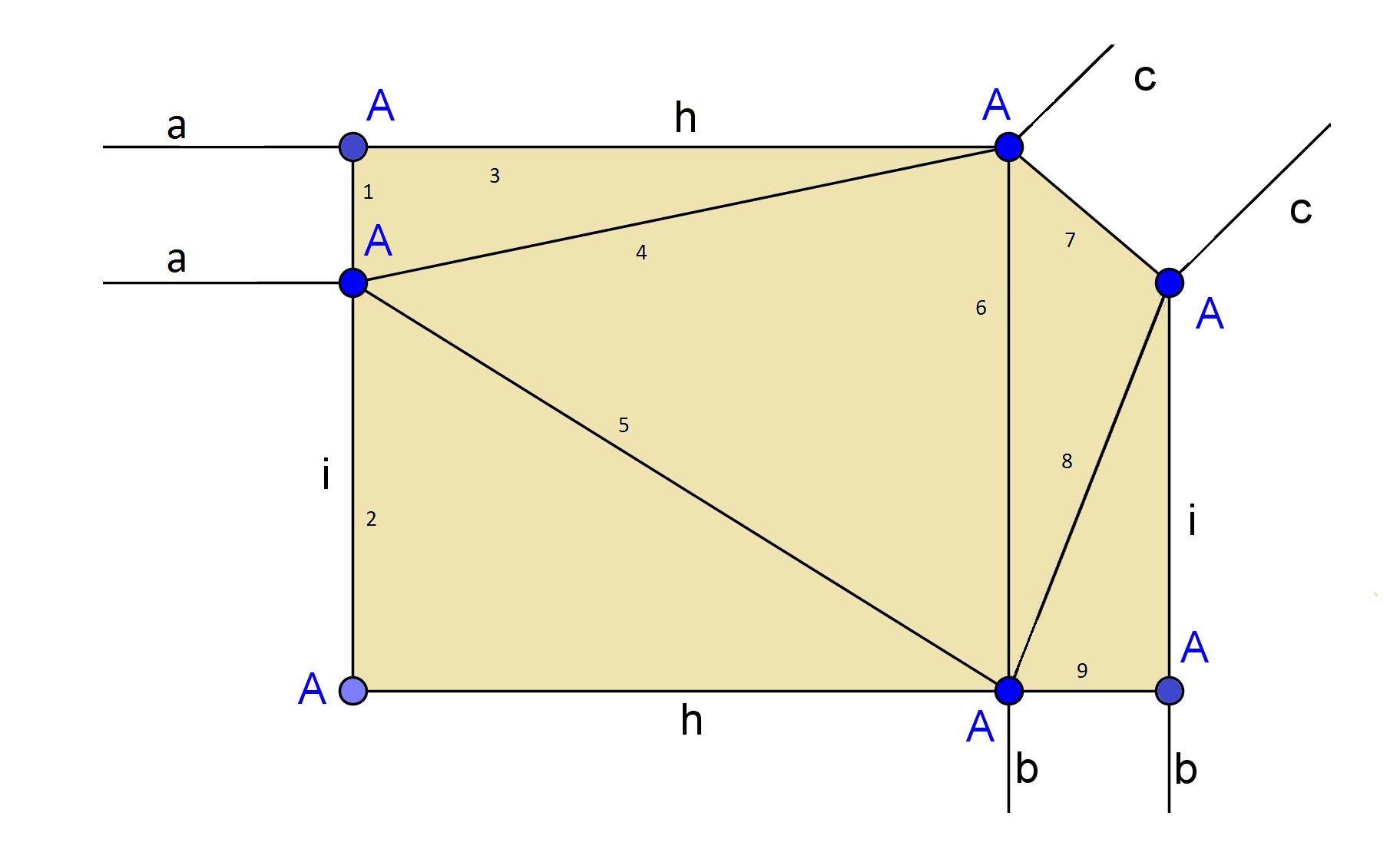}
\caption{A surface in $\mathcal{H}^{1}(3,-1^{3})$ with a maximal geodesic arc system including 9 saddle connections drawn on it}
\end{figure}

\begin{prop} Let $|A|$ be the number of edges of a maximal geodesic arc system for a flat surface with poles of higher order $(X,\phi)$ of genus $g$ belonging to $\mathcal{H}^{k}(a_1,\dots,a_n,-k^{p}) \setminus \Sigma$, then:\newline
- if $k$ is odd, we have $|A| = 6g-6+3n+2p$,\newline
- if $k$ is even and $p\geq2$, we have $|A| = 6g-6+3n+2p$,\newline
- if $k$ is even and $p=1$, we have $|A| \geq 6g+3n-5$.
\end{prop}

\begin{proof}
Domains of poles of order $k$ are semi-infinite cylinders and two consecutive saddle connections of the boundary share an angle of $\pi$. Outside the discriminant, if there are several saddle connections in the boundary of the domain of a pole of order $k$, then they are parallel saddle connections. Otherwise, the boundary is a single closed saddle connection.\newline
Following Lemma 3.6, two saddle connections of the boundary are parallel if and only if they form the boundary of a surface with trivial holonomy. In this case, the geometry of the surface is very constrained, see Figure 20. It comprises two parts:\newline
- a unique domain of pole whose boundary is a pair of parallel saddle connections,\newline
- the interior of the core that is a connected domain with trivial holonomy.\newline

\begin{figure}
\includegraphics[scale=0.3]{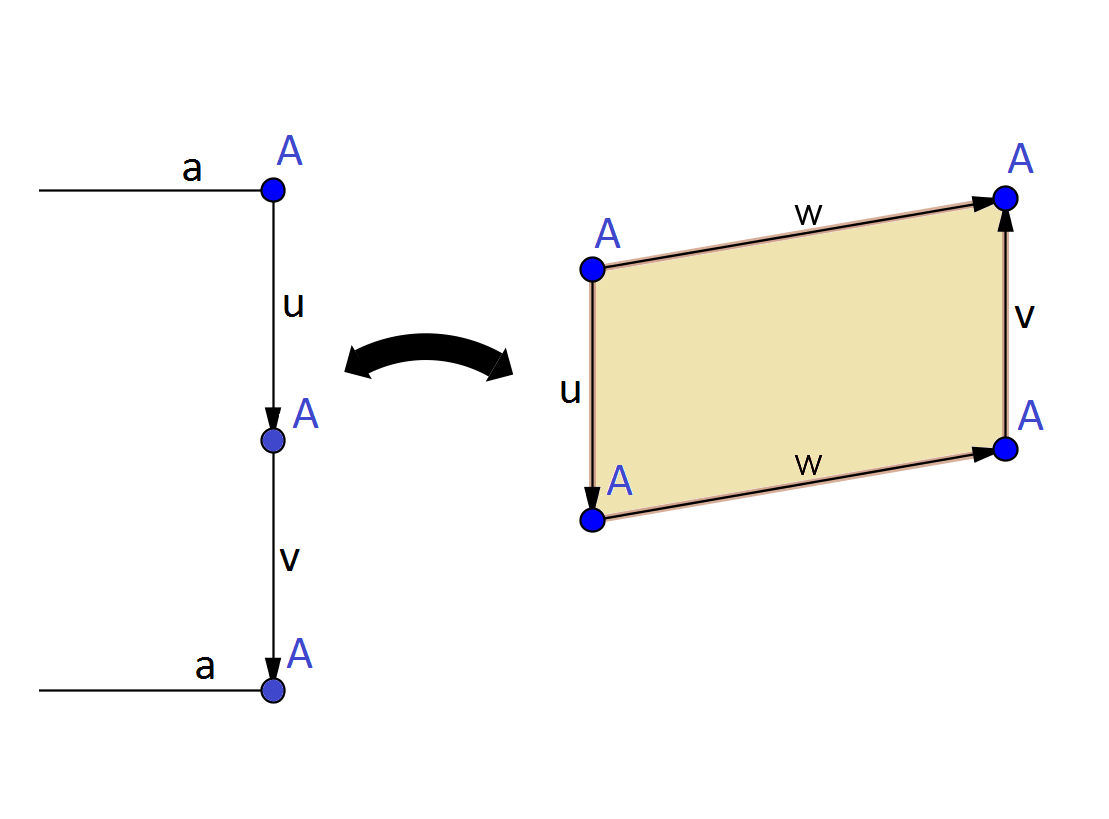}
\caption{A flat surface of $\mathcal{H}^{2}(2,-2)$.}
\end{figure}

Since the two saddle connections of the boundary share an angle of $\pi$, a rotation of angle $\pi$ should be realizable by the holonomy of the flat surface. Therefore, such a flat surface cannot exist when $k$ is odd.\newline
Outside the discriminant, when the domains of poles are bounded by a single closed saddle connection, we have $\beta=p$ and $|A|=6g-6+3n+2p$ (Lemma 7.4), see Figure 21.\newline
When $p=1$, the boundary of the unique pole can be formed by one or two saddle connections so we have $\beta\leq2$. Indeed, if there were three saddle connections in the boundary of a pole of order $k$ for a differential that does not belong to the discriminant, any consecutive pair of these saddle connections should be parallel. Following Lemma 3.6, these pairs of saddle connections should be the boundary of a subsurface. Such an assumption cannot hold for each pair of saddle connections. Thus we have, $|A|=6g-6+3n+3p-\beta \geq 6g+3n-5$.
\end{proof}

\begin{figure}
\includegraphics[scale=0.3]{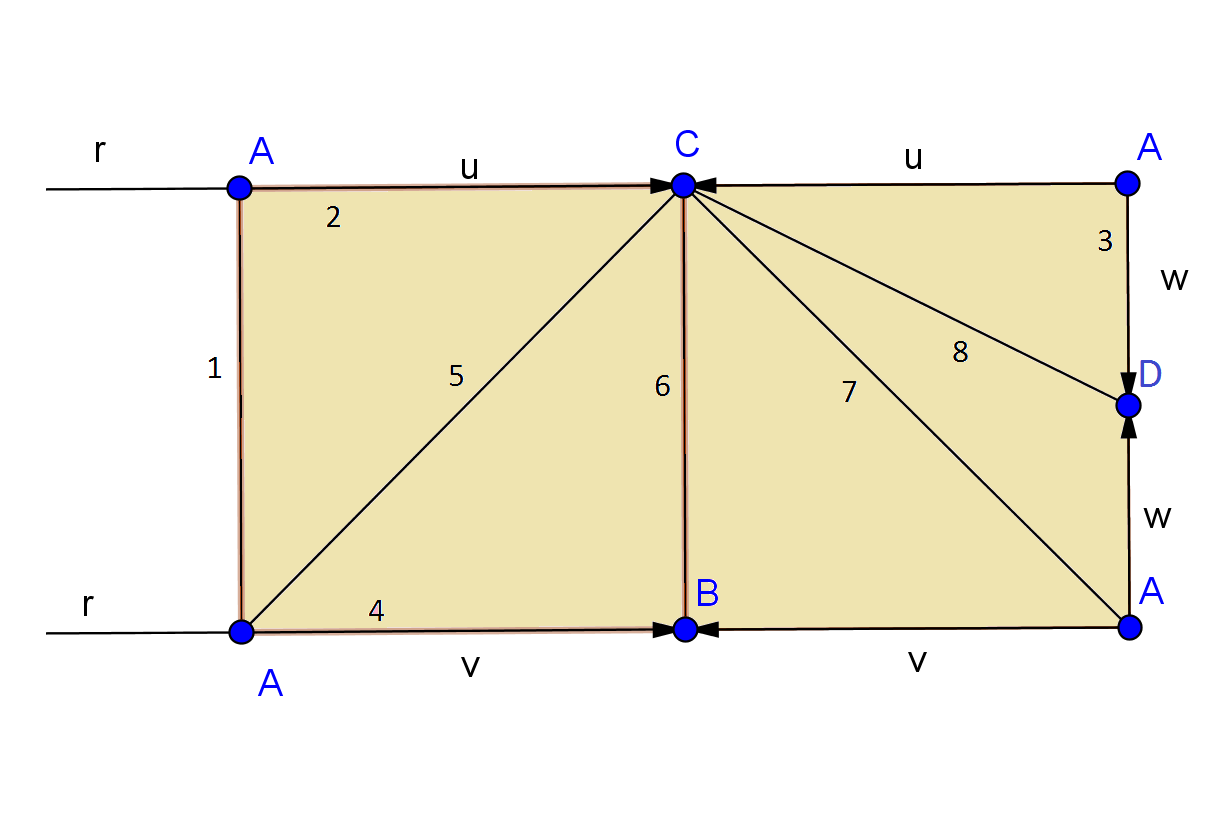}
\caption{A surface in $\mathcal{H}^{2}(1,-1^{3},-2)$ with a maximal geodesic arc system including 8 saddle connections drawn on it. The zero of order one is $A$ while the three poles of order one are $B$, $C$ and $D$.}
\end{figure}

\begin{rem}
For an abelian differential, the number of edges of a maximal geodesic arc system is $6g-6+3n$, see Proposition 2.1 in \cite{Vo}.
\end{rem}

Now that we have lower and upper bound on the maximal number of noncrossing saddle connections, it is natural to ask if every natural number in the interval is realized in the stratum. Some examples show it is not always true.

\begin{ex}
Let $(X,\omega)$ be a translation surface with poles of $\mathcal{H}^{1}(1^{2},-2^{2})$ and $|A|$ the number of edges of a maximal geodesic arc system. There are two types of surfaces:\newline
- $|A|=2$, $core(X)$ is degenerate and $|SC|=2$, see Figure 22.\newline
- $|A|=4$, $core(X)$ is a finite volume cylinder and $|SC|=+\infty$, see Figure 23.
\end{ex}

\begin{figure}
\includegraphics[scale=0.3]{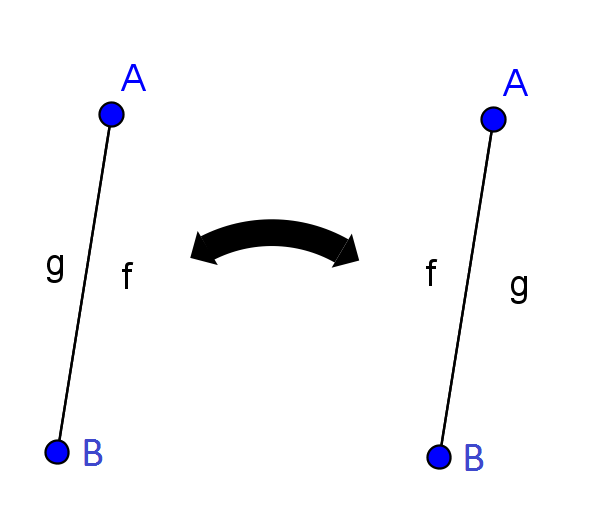}
\caption{A surface in $\mathcal{H}^{1}(1^{2},-2^{2})$ with degenerate core. It is formed by two flat planes with slits identified.}
\end{figure}

\begin{figure}
\includegraphics[scale=0.3]{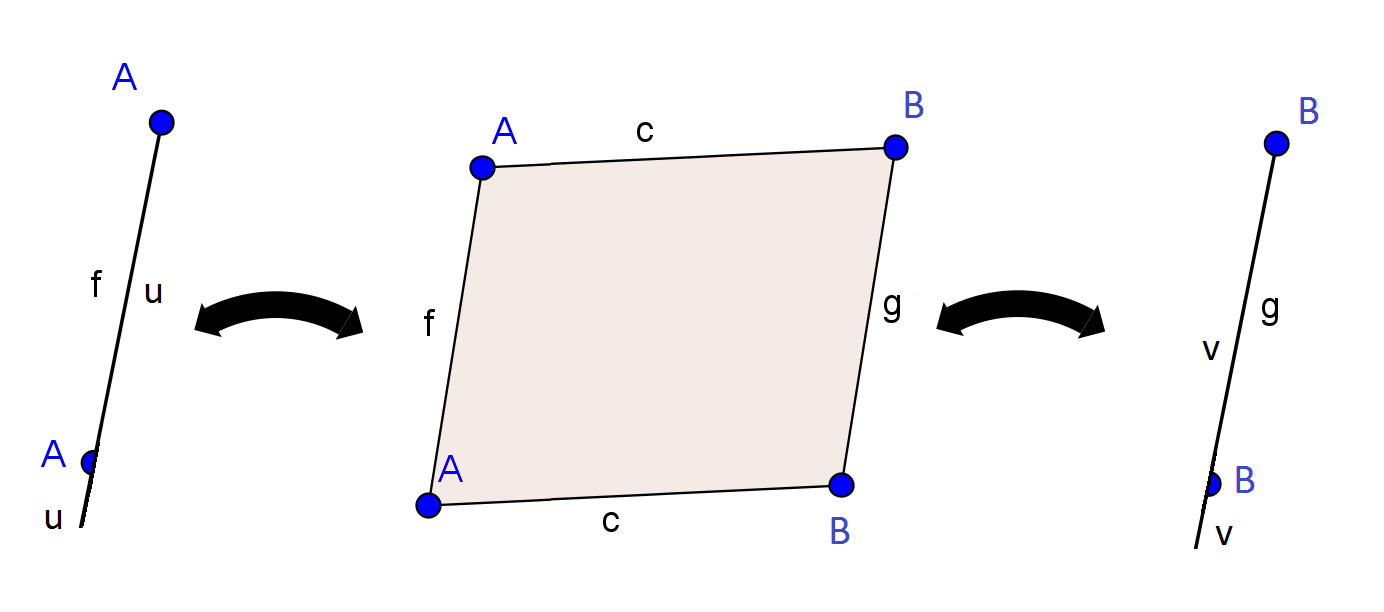}
\caption{A surface in $\mathcal{H}^{1}(1^{2},-2^{2})$ whose core is a cylinder. The two domains of poles are flat planes with slits identified with the boundary components of the cylinder.}
\end{figure}

\begin{proof}
Proposition 7.1 and Corollary 7.5 imply that $2 \leq |A| \leq 4$. $|A|=2$ if and only if $core(X)$ is degenerate. In this case, $|A|=|SC|=2$.\newline
\newline
Lemma 7.4 gives restrictions on the maximal number of noncrossing saddle connections:
$$ |A| = 6g-6+3n+3p - 2\sum \limits_{j=1}^p d_j - \sum \limits_{j=1}^p e_j$$

If $|A|=4$, then the boundary of each of the two domains of poles is a single closed saddle connection and their holonomy vectors are opposed. Therefore, $core(X)$ is not degenerate and is an invariant component for the flow in this direction. So $core(X)$ is a cylinder and $X$ has infinitely many saddle connections.\newline
\newline
We show now that the case $|A|=3$ is not possible in this stratum. Following the same equation, $core(X)$ would be an ideal triangle. One of its edges would be a closed saddle connection joining conical singularity $x_{1}$ to itself. This edge would be the whole boundary of the domain of pole $y_{1}$. Consequently, the domain of pole $y_{2}$ would be bounded by the two other sides of the ideal triangle. These sides would be saddle connections joining $x_{1}$ to $x_{2}$. The angle of $x_{1}$ in the domain of pole $y_{1}$ is $3\pi$ (using the topological order of the loop defined by the closed saddle connection). Conical singularity $x_{1}$ also belongs to the boundary of the domain of $y_{2}$ and since the core is convex, the angle is at least $\pi$ (Proposition 4.11). There is also a contribution of two interior angles of the triangles. Therefore, the total angle of $x_{1}$ is strictly bigger than $4\pi$. So $|A|=3$ is not possible.\newline
\end{proof}

\section{Dichotomy between chambers}

The main theorem of this section is only about strata of $1$-forms. Following Corollary 2.4, all surfaces have finitely many saddle connections in a few strata (when $g=\kappa=0)$. On the contrary, in the other ones, some surfaces have infinitely many saddle connections. The previous theorems about the dynamics in translation surfaces with poles give information to distinguish the different families of strata.\newline
\newline
The following examples show that different surfaces of the same stratum can have either finite or infinite number of saddle connections.

\begin{ex} In $\mathcal{H}^{1}(2,-2)$, some surfaces have a finite volume cylinder so have an infinite number of saddle connections and some have a finite number of saddle connections, see Figure 4 and Figure 16.

\begin{figure}
\includegraphics[scale=0.3]{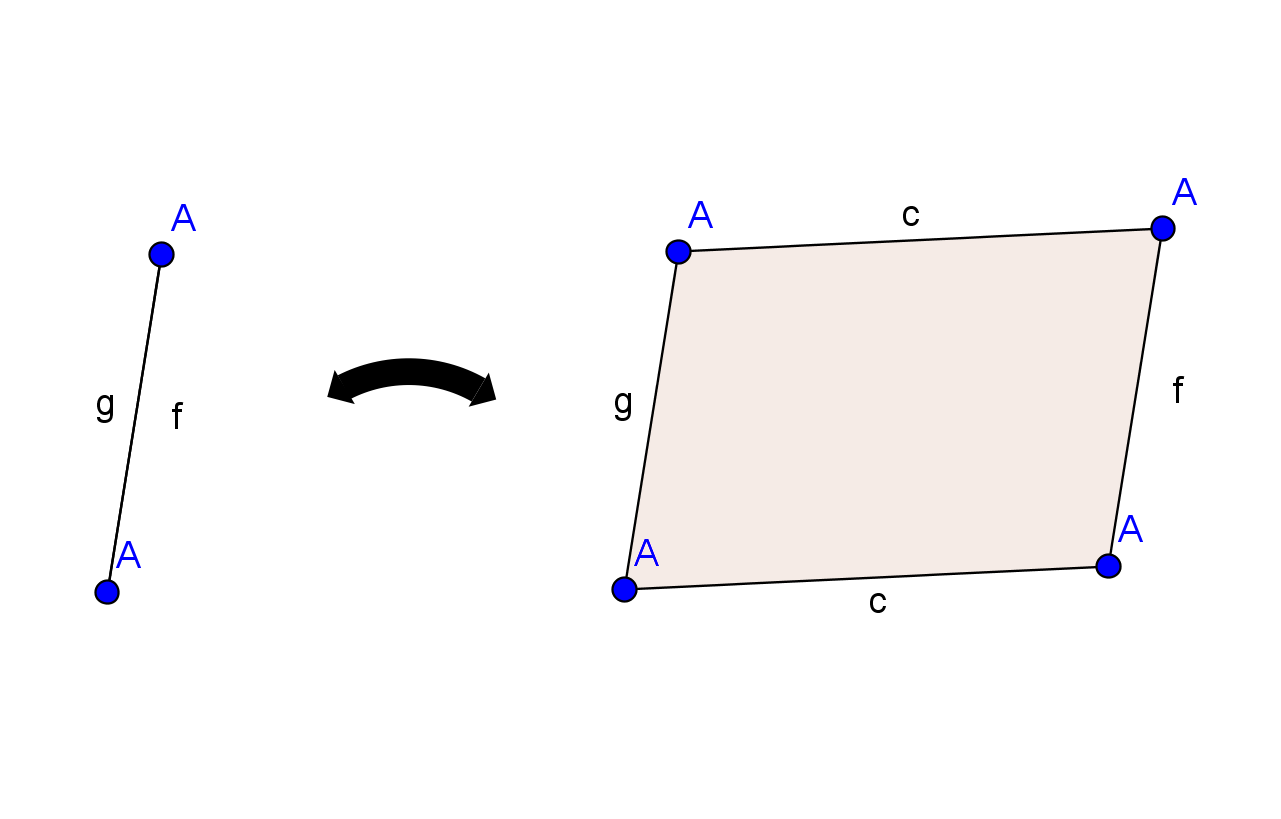}
\caption{A flat surface $\mathcal{H}^{1}(2,-2)$ that belongs to the chambers where the core is a cylinder.}
\end{figure}

\end{ex}

Previous results opened the way to a conditional upper bound on the number of saddle connections of a flat surface with poles of higher order. Corollary 5.13 proves that a flat surface where there is no finite volume cylinder has a finite number of saddle connections. Theorem 2.5 provides an explicit upper bound bound on the number of saddle connections for surfaces in chambers where where no flat surface has a finite volume cylinder. Theorem 2.5 is only about strata of $1$-forms.\newline
In flat surfaces corresponding to quadratic differentials, the orientation of saddle connection is not well defined so a saddle connection may cross another one in both directions. This complicates the counting of saddle connections and prevents obtaining a similar bound as in the case $k=1$. When $k\geq3$, saddle connections and periodic geodesics may be self-intersecting so the geometry of the cylinders is far more complicated. An analog result would need specific tools.\newline

Contrary to the case of strata where existence of surfaces with an infinite number of saddle connections has a combinatorial characterization (see Section 6), there is no easy characterization of chambers in terms of topological properties of the core. Existence of a cylinder of finite area implies existence of a closed simple loop whose topological degree is zero and whose homology class is nonzero. However, it is not proved that these conditions are sufficient.

\begin{proof}[Proof of Theorem 2.5]
Let $(X,\phi)$ be a flat surface in a chamber $\mathcal{C}$ of a stratum $\mathcal{H}$. We are going to show that a high number of saddle connections in $X$ implies existence of a cylinder for a surface in the same chamber as $(X,\phi)$.\newline

We first suppose there is a simple closed loop $\alpha$ (without self-intersection) that satisfies the following conditions:\newline
- $\alpha$ passes only through regular points of $\mathcal{I}\mathcal{C}(X)$,\newline
- the topological index of $\alpha$ is equal to zero,\newline
- $\alpha$ is the union of two subintervals of saddle connections and two turns (see Figure 25),\newline
- the angular defect (deviation from $\pi$) at the turns are strictly less than $\pi$.\newline

We replace curve $\alpha$ inside $X$ by a flat ribbon whose boundary is formed by two curves that are isometric to $\alpha$. If the ribbon is thick enough (compared to the length of $\alpha$), then there is a periodic geodesic inside it, see Figure 25. This periodic geodesic is a waist curve of a finite volume cylinder.\newline

\begin{figure}
\includegraphics[scale=0.3]{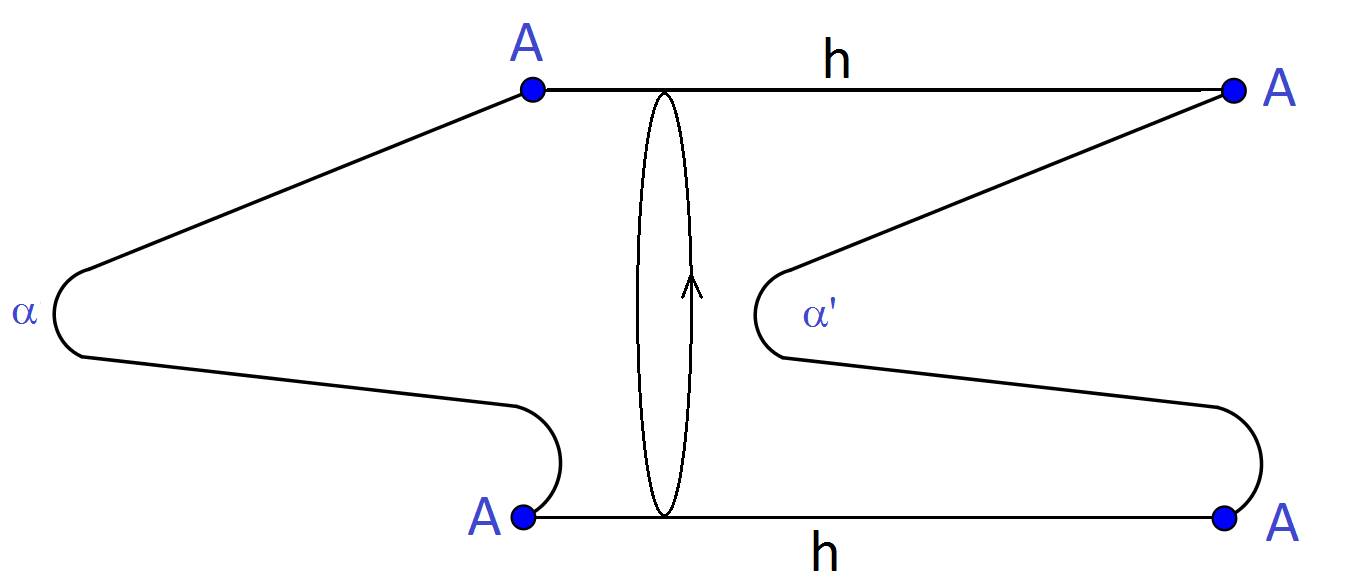}
\caption{A thick ribbon whose boundary is a pair of isometric curves $\alpha$ and $\alpha'$. It contains a family of closed geodesics.}
\end{figure}

The surgery is inserting a wide ribbon inside the interior of the core. Therefore, it does not change the structure of the core so the surface we get is still in chamber $\mathcal{C}$. As there is a finite volume cylinder, the new surface has infinitely many saddle connections (Corollary 5.13).\newline

Then, we prove an upper bound on the number of saddle connections of flat surfaces where there is no such closed loop $\alpha$.\newline

We consider a maximal geodesic arc system on $(X,\phi)$. Following Lemma 4.8, we get an ideal triangulation of $core(X)$. We denote by $t_1,\dots,t_s$ the numbers of ideal triangles in every connected component of the interior of $core(X)$. We have $t=\sum \limits_{{i=1}}^s t_i$. Following Lemma 4.10, we have $|A|=2g-2+n+p+t$.\newline

Every saddle connection that does not belong to the chosen maximal geodesic arc system go through a string of ideal triangles. There is no saddle connection that start and end in the same ideal triangle. Such a saddle connection either starts and ends in the same corner of the ideal triangle or intersect itself in the interior of the ideal triangle. Since the direction of a saddle connection is unambiguous if $k=1$, there is no such saddle connection.\newline

We consider a pair of distinct ideal triangles $r$ and $r'$ in the the same connected component of $\mathcal{I}\mathcal{C}(X)$. We suppose there are two distinct saddle connections $\gamma$ and $\gamma'$ that begin and end in these ideal triangles. Either the two saddle connections end in the same corner of the ideal triangle or they intersect each other in the interior of the ideal triangle. Following $\gamma$ from $r$ to $r'$ and $\gamma'$ from $r'$ to $r$ defines a path. If $\gamma$ and $\gamma'$ intersect in some interior points, we consider only portions of these saddle connections such that we get a simple closed loop $\alpha$. The angular defect at the two turning points are strictly less than $\pi$. Therefore, the topological index of $\alpha$ is zero. We get a simple loop that have all the needed properties.\newline
If there is no surface with an infinite number of saddle connection in the chamber, then for every couple distinct of ideal triangles in the same connected component of the core, there are at most one saddle connection that begins and ends in the given ideal triangles. Therefore, we have $|SC|\leq|A|+\dfrac{1}{2}\sum \limits_{{i=1}}^s t_{i}(t_{i}-1)$.\newline
\end{proof}

One can construct many surfaces by pasting boundaries of domains of poles on the sides of a convex k-gon. As there are $\dfrac{k(k-1)}{2}$ segments in a convex k-gon, a quadratic exponent is optimal.

\begin{ex}
In $\mathcal{H}^{1}(p-2,-1^{p})$ with $p\geq3$, there are surfaces with $\dfrac{p(p-1)}{2}$ saddle connections, see Figure 7.
\end{ex}

Corollary 2.6 provides a characterization of strata where every chamber is of bounded type. They are exactly irreducible strata. This result holds when $k=1$.

\begin{proof} [Proof of Corollary 2.6]
If $\mathcal{H}$ is an irreducible stratum, then every surface has finitely many saddle connections (Corollary 2.4). If there were a chamber of unbounded type in $\mathcal{H}$, there would be a surface with infinitely many saddle connections (Theorem 2.5). So every chamber is of bounded type.\newline
If $\mathcal{H}$ is a reducible stratum, then there is a surface $(X,\omega)$ in $\mathcal{H}$ with infinitely many saddle connections (Theorem 2.5). The locus in $\mathcal{H}$ where $|SC|=+\infty$ is an open set (Corollary 5.13). Besides, the discriminant has real codimension one (Proposition 4.14). Therefore, there is a at least a chamber where there is a surface with infinitely many saddle connections. This chamber is of unbounded type.\newline
As the number of saddle connections changes lower continuously in the stratum and the bound of Theorem 2.5 is satisfied in the chambers (in the case $k=1)$, the bound is satisfied in discriminant as well. Therefore, the bound is satisfied in the whole stratum. Following Corollary 7.7, we have $t\leq 4g-4+2n+p$ so we get the adequate bound from the bound of Theorem 2.5.
\end{proof}

For some strata, having an infinite number of saddle connections is a generic property and in particular is true for every chamber that is outside the discriminant.

\begin{prop}
Let $(X,\omega)$ be a translation surface with poles of genus $g \geq 1$ belonging to $\mathcal{H}^{1}(a_1,\dots,a_n,-1^2)\  \setminus\ \Sigma$. Then $(X,\phi)$ has an infinite number of saddle connections.
\end{prop}

\begin{proof}
Since poles are simple, outside the discriminant, the boundary of each domain of a pole is a unique closed saddle connection. Residues of the meromorphic differential at the poles are opposed and nonzero. $core(X)$ is not degenerate because otherwise, the surface would be the pasting of two half-cylinders on each other, that is a surface of genus zero with two poles and no conical singularities.\newline
If we consider the flow in the direction of the residues, $core(X)$ splits into a finite number of finite volume cylinders or minimal components where there are infinitely many saddle connections.\newline
\end{proof}

In very specific chambers, every surface has infinitely many saddle connections. Such a proposition holds for $1$-forms as well as for quadratic differentials.

\begin{prop}
Let $\mathcal{C}$ be a chamber of a stratum $\mathcal{H}^{k}$ of $1$-forms ($k=1$) or quadratic differentials ($k=2$) such that the boundary of one connected component is a unique closed saddle connection or a pair of parallel saddle connections. Then every surface $(X,\phi)$ in $\mathcal{C}$ has infinitely many saddle connections. Chamber $\mathcal{C}$ is of unbounded type.
\end{prop}

\begin{proof}
We consider a surface $(X,\phi)$ in $\mathcal{C}$. The holonomy vectors of the two saddle connections of the boundary of $\mathcal{U}$ are opposite to one another. Thus, if we consider the flow in the direction of the holonomy vector of the saddle connections, trajectories of $\mathcal{U}$ are either minimal or periodic because they cannot leave $\mathcal{U}$ (trajectories are parallel to the boundary). Therefore, there are infinitely many saddle connections in $X$ (Corollary 5.13).\newline
\end{proof}

\nopagebreak
\vskip.5cm
\end{document}